\documentclass[12pt,final]{amsart}
\usepackage{amssymb,euscript,nicefrac}
\usepackage{ytableau}
\usepackage{amscd,txfonts}
\usepackage[notref,notcite]{showkeys}
\usepackage{latexsym,todonotes}
\usepackage{xr,xr-hyper}
\usepackage{epsfig,pxfonts}
\usepackage{amsfonts}
\usepackage{xypic}
\usepackage{bbm,ifpdf}
\ifpdf
  \usepackage{pdfsync,hyperref}
\fi
\usepackage[margin=3cm,footskip=25pt,headheight=21pt]{geometry}
\usepackage{fancyhdr,mathrsfs}
\newtheorem{theorem}{Theorem}[section]
\newtheorem{lemma}[theorem]{Lemma}
\newtheorem{proposition}[theorem]{Proposition}
\newtheorem{corollary}[theorem]{Corollary}

{
\theoremstyle{plain}
\newtheorem{definition}[theorem]{Definition}

\newtheorem{remark}[theorem]{Remark}

}

\renewcommand{\theequation}{\arabic{section}.\arabic{equation}}
\newcounter{subeqn}
\renewcommand{\thesubeqn}{\theequation\alph{subeqn}}
\newcommand{\subeqn}{%
  \refstepcounter{subeqn}
  \tag{\thesubeqn}
}
\makeatletter
\@addtoreset{subeqn}{equation}
\newcommand{\newseq}{%
  \refstepcounter{equation}
}

\newcommand{\nc}{\newcommand}
\nc{\cat}{\mathcal{V}}
\nc{\func}{\EuScript{T}}
\nc{\res}{\operatorname{res}}

\renewcommand{\dim}{\operatorname{dim}}

\nc{\slehat}{\mathfrak{\widehat{sl}}_e}
\nc{\sllhat}{\mathfrak{\widehat{sl}}_\ell}
\nc{\glehat}{\mathfrak{\widehat{gl}}_e}
\nc{\slnhat}{\mathfrak{\widehat{sl}}_n}
\nc{\glnhat}{\mathfrak{\widehat{gl}}_n}
\nc{\eE}{\EuScript{E}}
\newcommand{\arxiv}[1]{\href{http://arxiv.org/abs/#1}{\tt arXiv:\nolinkurl{#1}}}

\nc{\eF}{\EuScript{F}}
\nc{\fF}{\mathfrak{F}}
\nc{\fE}{\mathfrak{E}}
\nc{\fI}{\mathfrak{I}}

\newcommand{\K}{\mathbbm{k}}

\newcommand{\Z}{\mathbb{Z}}
\newcommand{\Q}{\mathbb{Q}}
\nc{\Qlb}{\mathbb{\bar \Q}_\ell}
\nc{\Fq}{\mathbb{F}_q}
\nc{\Fqb}{\mathbb{\bar F}_q}

\nc{\walg}{W}

\newcommand{\R}{\mathbb{R}}

\newcommand{\C}{\mathbb{C}}

\nc{\KZ}{\mathsf{KZ}}

\newcommand{\la}{\leftarrow}

\newcommand{\sS}{\mathsf{S}} 
 
\newcommand{\sT}{\mathsf{T}}

\nc{\Bv}{\mathbf{v}}
\nc{\Gr}{\operatorname{Gr}}
  \nc{\Bw}{\mathbf{w}}
  \nc{\Bb}{\mathbf{b}}
\nc{\tU}{\mathcal{U}}
\nc{\Bu}{\mathbf{u}}
 \nc{\Fl}{\mathscr{F}\!\ell}
\nc{\Tr}{\operatorname{Tr}}
\nc{\sheafK}{\EuScript{K}}
\nc{\bmu}{\boldsymbol{\mu}}
\nc{\bpi}{\boldsymbol{\pi}}
\nc{\dwalg}{\mathbb{W}}
\nc{\dalg}{\mathbb{T}}
\nc{\aalg}{\mathbb{A}}
\nc{\alm}{\mathscr{A}}
\nc{\bra}{\mathscr{B}}
\nc{\bO}{\mathbb{O}}

\nc{\Kos}{\EuScript{K}}
\nc{\tilt}{\EuScript{T}}

\renewcommand{\la}{\lambda}

\newcommand{\Hom}{\operatorname{Hom}}
\newcommand{\RHom}{\mathbb{R}\operatorname{Hom}}

\newcommand{\cO}{\mathcal{O}}

\newcommand{\Ext}{\operatorname{Ext}}

\newcommand{\bS}{\mathbb{S}}

\newcommand{\excise}[1]{}

\newcommand{\End}{\operatorname{End}}

\newcommand{\mmod}{\operatorname{-mod}}
\newcommand{\gmod}{\operatorname{-gmod}}

\newcommand{\Lotimes}{\overset{L}{\otimes}}
\newcommand{\thetitle}{Tensor product algebras, Grassmannians and
  Khovanov homology}

\begin{document}

\renewcommand{\theitheorem}{\Alph{itheorem}}

\usetikzlibrary{decorations.pathreplacing,backgrounds,decorations.markings,calc,
shapes.geometric, patterns}
\tikzset{wei/.style={draw=red,double=red!40!white,double distance=1.5pt,thin}}

\noindent {\Large \bf 
\thetitle}
\bigskip\\
{\bf Ben Webster}\footnote{Supported by the NSF under Grant DMS-1151473.}\\
Department of Mathematics, University of Virginia, Charlottesville, VA
\bigskip\\
{\small
\begin{quote}
\noindent {\em Abstract.}
We discuss a new perspective on Khovanov homology, using
categorifications of tensor products.  While in many ways
more technically demanding than Khovanov's approach (and its extension
by Bar-Natan), this has distinct advantage of directly connecting
Khovanov homology to a categorification of $(\mathbb{C}^2)^{\otimes
  \ell}$, and admitting a direct generalization to other Lie
algebras.  

While the construction discussed is a special case of that given in
earlier work of the author,
this paper contains new results about the 
case of $\mathfrak{sl}_2$ showing an explicit connection to
Bar-Natan's approach to Khovanov homology, to the geometry of
Grassmannians, and to the categorified Jones-Wenzl projectors of
Cooper and Krushkal.  In particular, we show that the colored Jones
homology defined by our approach coincides with that of Cooper and
Krushkal.
\end{quote}
}
\bigskip

\section{Introduction}
\label{sec:introduction}

{\small \it 
\begin{quotation}
Man is a knot, a web, a mesh into which relationships are tied.

\hfill --Antoine Saint-Exupery (1942)
\end{quotation}
}

Khovanov homology has proven one of the most remarkable constructions
of recent years, and has stimulated a great deal of work in the field
of knot homology.  Khovanov homology is a categorification of
the Jones polynomial \cite{Jon87}, which is a special case (for the
defining representation $\mathbb{C}^2$ of $\mathfrak{sl}_2$) of the
Reshetikhin-Turaev invariants attached to representations of simple
Lie algebras\footnote{The most common construction of these invariants
uses deformations of these representations to modules over the quantum
group associated to a Lie algebra.  Thoughout, we'll use the name of a
Lie algebra, usually $\mathfrak{sl}_2$, to also refer to other constructions
based on its Cartan matrix, like quantum groups.}.  This leads to the natural question, which has attracted a great
deal of attention, of whether the Reshetikhin-Turaev invariants for
other Lie algebras and representations have categorifications like
Khovanov homology; a general construction of such invariants was given
by the author in \cite{Webmerged}, building on a decade's worth of work
by many authors.

From the original construction of Khovanov homology, it's not easy to
see why this should be possible.  The Reshetikhin-Turaev construction
is based on the ribbon structure on the tensor category of
$U_q(\mathfrak{g})$, but the early definitions of
Khovanov homology had no clear connection to tensor products of
representations of $U_q(\mathfrak{sl}_2)$.  Our intent in this note is to
sketch out a new construction of Khovanov homology which can be
generalized to other representations of other Lie algebras.  

This construction is a special case of that given in \cite{Webmerged};
 following that paper, it will first be described in Section
 \ref{sec:tens-prod-algebr} in purely algebraic language, introducing
 certain diagrammatic algebras $T^\ell$ ({\it \`a la} Khovanov-Lauda \cite{KLI}) whose representation categories categorify
 the tensor power $(\C^2)^{\otimes \ell}$ of the defining
 representation of $\mathfrak{sl}_2$
 or its quantum analogue (in a sense that we will make more precise).  
The results of that section are with a few exceptions special cases of
those of \cite{Webmerged}, and many of the proofs will be farmed out.

Another part of our aim is also to describe the relationship of this
construction with geometry, which is discussed in Section
\ref{sec:geom-grassm}. In the case where $\mathfrak{g}=\mathfrak{sl}_2$,
the subject of this paper, this
underlying geometry is that of Grassmannians; for higher rank groups,
it is the geometry of Nakajima quiver varieties (see
\cite{Webcatq,Webqui}).  More specifically, the algebra $T^\ell$ is
isomorphic to a convolution algebra defined using the Grassmannian and
certain related varieties.  This geometry provides a motivation for
understanding these algebras, and a more systematic way of thinking
about their definition, as well as relating this work to more
traditional geometric representation theory.  In particular, it shows
that the algebras $T^\ell$ are Koszul dual to the generalized arc
algebras of Stroppel \cite{Str09} (Theorem \ref{th:arc-Koszul}); thus our construction of Khovanov
homology is matched by Koszul duality with that of Khovanov \cite{Kho02} and
Stroppel \cite{StDuke}.  While a number of related geometric results
have appeared in the literature (for example in \cite{WebwKLR}), this
precise connection seems not to have been written before.

Finally, in the last section, we will give a short account of how to
precisely match up the construction we have given with Bar-Natan's
construction of Khovanov homology using a quotient of the cobordism
category.  As shown by Chatav \cite{Chatav}, Bar-Natan's construction
\cite{BarN05} applied to cobordisms
between flat tangles (what is often called the {\it Temperley-Lieb
  2-category}) can be interpreted as a 2-category which acts on the
derived categories of modules over $T^\ell$ (for all
$\ell$). Combining these results, we arrive at our main theorem:
\begin{theorem}\label{Khovanov}
  The knot invariants defined in \cite{Webmerged} for the
  representation $\C^2$ of $\mathfrak{sl}_2$ agree with Khovanov
  homology, up to a reindexing of gradings: Bar-Natan's internal
  grading agrees with ours, but his homological grading is the sum of
  our internal and homological grading.
\end{theorem}
We can also interpret the categorified Jones-Wenzl projector of Cooper
and Krushkal \cite{CoKr} as projection onto a natural subcategory in
our picture.  
\begin{theorem}
  The knot invariants defined in \cite{Webmerged} for the higher
  dimensional
  representations of $\mathfrak{sl}_2$ agree with those of \cite{CoKr}
  based on the categorified Jones-Wenzl projector.
\end{theorem}

\section{Tensor product algebras of $\mathfrak{sl}_2$}
\label{sec:tens-prod-algebr}

{\small \it 
\begin{quotation}
I see but one rule: to be clear. If I am not clear, all my world crumbles to nothing.

\hfill --Stendhal (1840)
\end{quotation}
}

\subsection{Stendhal diagrams}
\label{sec:stendhal-diagrams}

We wish to define the algebra $T^\ell$ as discussed in the introduction.
\begin{definition}
  A {\bf Stendhal diagram} is an arbitrary finite number of smooth red and
  black curves in $\R\times [0,1]$ subject to the rules:
  \begin{itemize}
\item The endpoints of the curves must lie at distinct points of $\R\times \{0,1\}$.
  \item These curves must be oriented downward at each point. In
    particular, they have no local minima or maxima.
  \item Black curves can intersect other black curves and red curves,
    but pairs of red curves are not allowed to intersect.
  \item This collection of curves has no tangencies or triple (or higher)
    intersection points.
  \end{itemize}
  Each black strand can additionally carry dots that don't occur at
  crossing points; we'll represent a group of a number of dots as a
  single dot with that number next to it.  We'll consider these configurations
  up to isotopy that doesn't change any of these conditions (including
  isotopy of dots avoiding crossings).
\end{definition}

  Here are two examples of Stendhal diagrams:
  \begin{equation}
a=
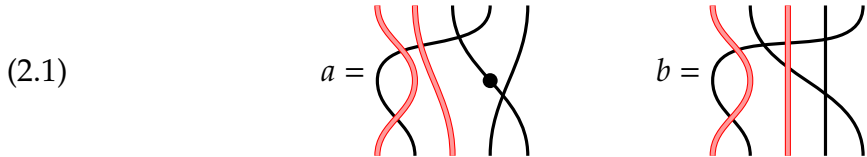
\begin{tikzpicture}[baseline,very thick]
  \draw (-.5,-1) to[out=90,in=-90] (-1,0) to[out=90,in=-90](.5,1);
  \draw (.5,-1) to[out=90,in=-90] (1,1);
  \draw  (1,-1) to[out=90,in=-90] 
  node[midway,circle,fill=black,inner sep=2pt]{} (0,1);
  \draw[wei] (-1, -1) to[out=90,in=-90] (-.5,0) to[out=90,in=-90] (-1,1);
  \draw[wei] (0,-1) to[out=90,in=-90](-.5,1);
\end{tikzpicture}\qquad \qquad 
b=
\begin{tikzpicture}[baseline,very thick]
  \draw (-.5,-1) to[out=90,in=-90] (-1,0) to[out=90,in=-90](1,1);
  \draw (.5,-1) to[out=90,in=-90] (.5,1);
  \draw  (1,-1) to[out=90,in=-90] (-.5 ,1);
  \draw[wei] (0,-1) to[out=90,in=-90] (0,1);
  \draw[wei] (-1, -1) to[out=90,in=-90] (-.5,0) to[out=90,in=-90] (-1,1);
\end{tikzpicture}\label{eq:examples}
\end{equation}

Stendhal diagrams have a product structure given by letting $ab$ be
given by stacking $a$ on top\footnote{Thus, we read diagrams from
  bottom to top; we will usually read diagrams left to right.} of $b$, and attempting to attach strands
while preserving colors.  Since we only consider these diagrams up to
isotopy, only the order of red and black strands is relevant.  If this is not possible, then we simply say
that the composition is $0$.  For example:
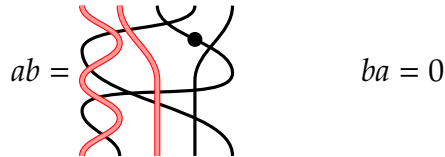
\[
ab=
\begin{tikzpicture}[baseline,very thick,yscale=.5]
  \draw (-.5,-2) to[out=90,in=-90] (-1,-.8)
  to[out=90,in=-90](1,.2) to[out=90,in=-90] node[midway,circle,fill=black,inner sep=2pt]{}  (0,2);
  \draw (.5,-2) to[out=90,in=-90] (.5,0) to[out=90,in=-90] (1,2);
  \draw  (1,-2) to[out=90,in=-90] (-1,.8) to[out=90,in=-90] (.5,2);
  \draw[wei] (0,-2) to[out=90,in=-90]
  (0,0) to[out=90,in=-90]
  (-.5,2);
  \draw[wei] (-1, -2) to[out=90,in=-90] 
  (-.5,-1) to[out=90,in=-90] (-1,0) to[out=90,in=-90] (-.5,1)
  to[out=90,in=-90]   (-1,2);
\end{tikzpicture}\qquad \qquad ba=0
\]

A more explicit way of encoding the pattern of red and black strands
in a slice, if we have $\ell$ red strands and
$k$ black strands, is to define a map $\kappa\colon [1,\ell]\to
[0,k]$ attached to any generic
horizontal slice of a Stendhal diagram (i.e. one which avoids all
intersection points) sending $h$ to the number of
black strands left of the $h$th red strand (counted from left).  We
must have that the function attached to the top of $b$ ($y=1$)
coincides with that attached
to the bottom of $a$ ($y=0$), or the product is 0.

\begin{definition}\label{def:degree}
  The {\bf degree} of a Stendhal diagram is an integer assigned to
  each diagram, given by the sum of the number of red/black crossings
  plus twice the number of dots, minus twice the number of black/black
  crossings.  Note that this number is additive under composition.
\end{definition}

Fix a field $\K$ and an integer $n$.
\begin{definition}
  Let $T^\ell_n$ be the graded algebra spanned over $\K$ by Stendhal
  diagrams with
  $\ell$ red strands and $k=(\ell-n)/2$ black strands\footnote{
By convention, if $k$ is not a non-negative integer, then
$T^\ell_n=\{0\}$.}, graded as in Definition \ref{def:degree}, modulo
the homogeneous local relations:
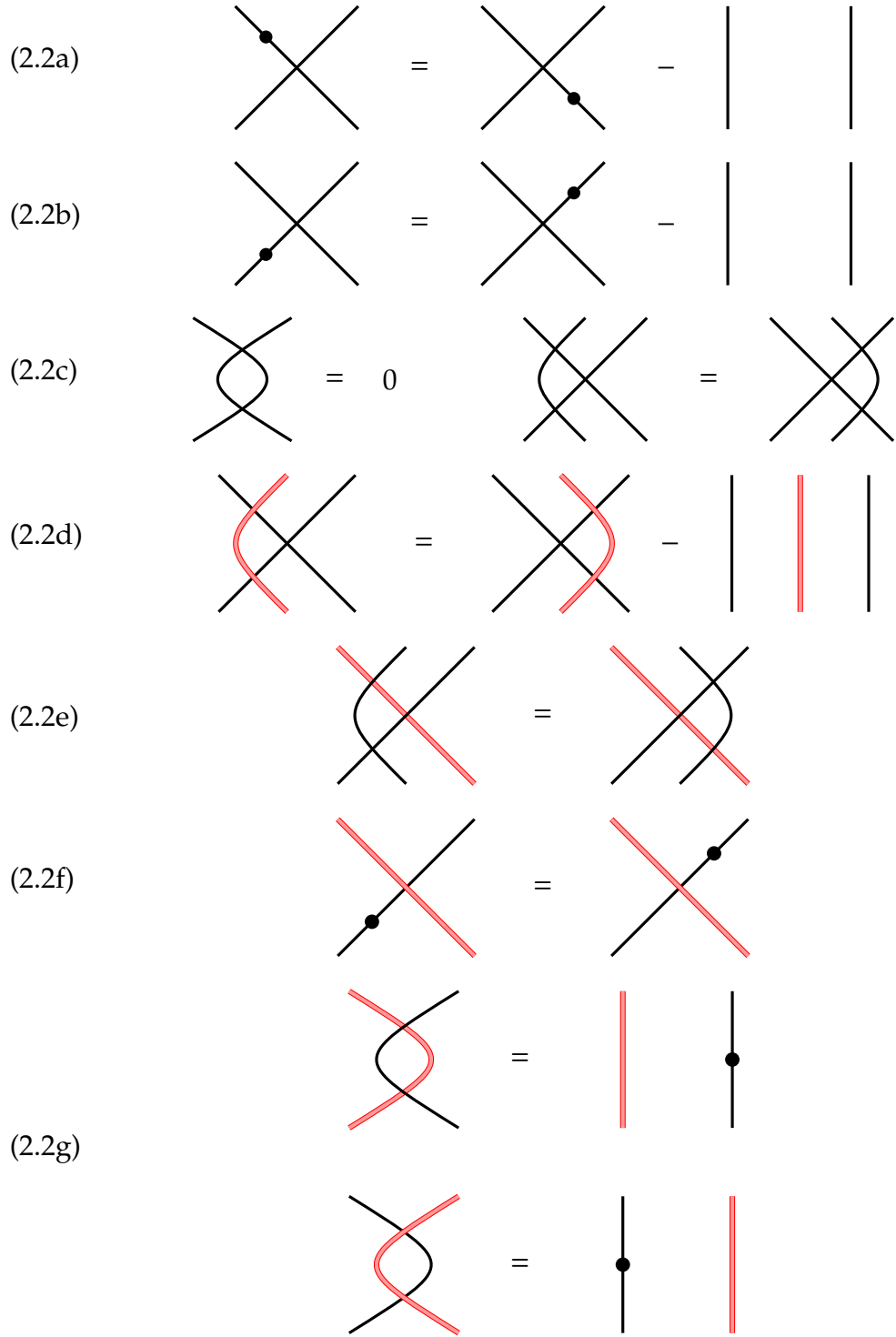
\newseq
\begin{equation*}\subeqn \label{first-nilhecke}
    \begin{tikzpicture}[scale=.9,baseline]
      \draw[very thick](-4,0) +(-1,-1) -- +(1,1); \draw[very thick](-4,0) +(1,-1) -- +(-1,1); \fill (-4.5,.5) circle (3pt);
      \node at (-2,0){=}; \draw[very thick](0,0) +(-1,-1) -- +(1,1); \draw[very thick](0,0) +(1,-1) --
      +(-1,1); \fill (.5,-.5) circle (3pt);
      \node at (2,0){$-$}; \draw[very thick](4,0) +(-1,-1) -- +(-1,1)
    ; \draw[very thick](4,0) +(1,-1) --
      +(1,1);
    \end{tikzpicture}
  \end{equation*}
 \begin{equation*}\subeqn
    \begin{tikzpicture}[scale=.9,baseline]
      \draw[very thick](-4,0) +(-1,-1) -- +(1,1); \draw[very thick](-4,0) +(1,-1) -- +(-1,1); \fill (-4.5,-.5) circle (3pt);
      \node at (-2,0){=}; \draw[very thick](0,0) +(-1,-1) -- +(1,1)
    ; \draw[very thick](0,0) +(1,-1) --
      +(-1,1); \fill (.5,.5) circle (3pt);
      \node at (2,0){$-$}; \draw[very thick](4,0) +(-1,-1) -- +(-1,1)
     ; \draw[very thick](4,0) +(1,-1) --
      +(1,1);
    \end{tikzpicture}
  \end{equation*}
  \begin{equation*}\subeqn\label{black-bigon}
    \begin{tikzpicture}[very thick,scale=.9,baseline]
      \draw (-2.8,-1) .. controls (-1.2,0) ..  (-2.8,1)
     ; \draw (-1.2,-1) .. controls (-2.8,0) ..  (-1.2,1); \node at (-.5,0)
      {=}; \node at (0.4,0) {$0$};
    \end{tikzpicture}
\qquad \qquad 
    \begin{tikzpicture}[very thick,scale=.9,baseline]
      \draw (-3,0) +(1,-1) -- +(-1,1); \draw
      (-3,0) +(-1,-1) -- +(1,1); \draw
      (-3,-1) .. controls (-4,0) ..  (-3,1); \node at (-1,0) {=}; \draw (1,0) +(1,-1) -- +(-1,1)
    ; \draw (1,0) +(-1,-1) -- +(1,1)
     ; \draw (1,-1) .. controls
      (2,0) ..  (1,1); 
    \end{tikzpicture}
  \end{equation*}
\begin{equation*}\subeqn
    \begin{tikzpicture}[very thick,baseline]\label{red-triple-correction}
      \draw (-3,0)  +(1,-1) -- +(-1,1);
      \draw (-3,0) +(-1,-1) -- +(1,1);
     \draw[wei]
      (-3,-1) .. controls (-4,0) ..  (-3,1);
      \node at (-1,0) {=};
      \draw (1,0)  +(1,-1) -- +(-1,1);
      \draw (1,0) +(-1,-1) -- +(1,1);
      \draw[wei] (1,-1) .. controls
      (2,0) ..  (1,1);   
\node at (2.6,0) {$- $};
      \draw (4.5,0)  +(1,-1) -- +(1,1);
      \draw (4.5,0) +(-1,-1) -- +(-1,1);
      \draw[wei] (4.5,0) +(0,-1) --  +(0,1);
 \end{tikzpicture}
  \end{equation*}
\begin{equation*}\subeqn\label{dumb}
    \begin{tikzpicture}[very thick,baseline=2.85cm]
      \draw[wei] (-3,3)  +(1,-1) -- +(-1,1);
      \draw (-3,2) .. controls (-4,3) ..  (-3,4);
      \draw (-3,3) +(-1,-1) -- +(1,1);
      \node at (-1,3) {=};
      \draw[wei] (1,3)  +(1,-1) -- +(-1,1);
 \draw (1,2) .. controls (2,3) ..  (1,4);
      \draw (1,3) +(-1,-1) -- +(1,1);    \end{tikzpicture}
  \end{equation*}
\begin{equation*}\subeqn\label{red-dot}
    \begin{tikzpicture}[very thick,baseline]
  \draw(-3,0) +(-1,-1) -- +(1,1);
  \draw[wei](-3,0) +(1,-1) -- +(-1,1);
\fill (-3.5,-.5) circle (3pt);
\node at (-1,0) {=};
 \draw(1,0) +(-1,-1) -- +(1,1);
  \draw[wei](1,0) +(1,-1) -- +(-1,1);
\fill (1.5,.5) circle (3pt);
    \end{tikzpicture}
  \end{equation*}
  \begin{equation*}\subeqn\label{cost}
  \begin{tikzpicture}[very thick,baseline=1.6cm]
    \draw (-2.8,-1) .. controls (-1.2,0) ..  (-2.8,1);
       \draw[wei] (-1.2,-1) .. controls (-2.8,0) ..  (-1.2,1);
           \node at (-.3,0) {=};
    \draw[wei] (2.8,0)  +(0,-1) -- +(0,1);
       \draw (1.2,0)  +(0,-1) -- +(0,1);
       \fill (1.2,0) circle (3pt);
          \draw[wei] (-2.8,2) .. controls (-1.2,3) ..  (-2.8,4);
  \draw (-1.2,2) .. controls (-2.8,3) ..  (-1.2,4);
           \node at (-.3,3) {=};
    \draw (2.8,3)  +(0,-1) -- +(0,1);
       \draw[wei] (1.2,3)  +(0,-1) -- +(0,1);
       \fill (2.8,3) circle (3pt);
  \end{tikzpicture}
\end{equation*}
\begin{equation*}
  \subeqn\label{violating} 
 \begin{tikzpicture}[very thick,baseline]
    \draw (-2.8,-1)--(-2.8,1);
    \node at (-1.8,0) {$\cdots$};
           \node at (-.3,0) {=};
    \node at (1.2,0) {$0$};
  \end{tikzpicture}
\end{equation*}
This last equation perhaps requires a little explanation. It should be
interpreted as saying that we set to 0 any Stendhal diagram with a
generic slice
$y=a$ where the leftmost strand is black, that is, where
$\kappa(1)>0$.

We let $T^\ell\cong \oplus_n T^\ell_n$.  If we consider the same span
of diagrams modulo only the relations
(\ref{first-nilhecke}--\ref{cost}), omitting (\ref{violating}), then
we denote the corresponding algebras $\tilde{T}^\ell_n$ and $\tilde{T}^\ell$.
\end{definition}
For example, the diagrams $a$ and $b$ defined in \eqref{eq:examples}
are both 0 in $T^\ell$, by the relation (\ref{violating}); in
$\tilde{T}^\ell$, they are not 0, but can be simplified to:
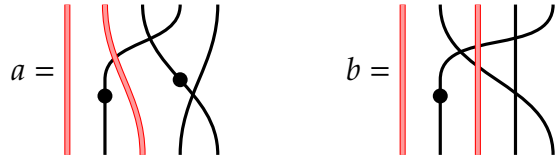
\[a=
\begin{tikzpicture}[baseline,very thick]
  \draw (-.5,-1) to[out=90,in=-90]  node[pos=.8,circle,fill=black,inner sep=2pt]{}  (-.5,0)  to[out=90,in=-90](.5,1);
  \draw (.5,-1) to[out=90,in=-90] (1,1);
  \draw  (1,-1) to[out=90,in=-90] 
  node[midway,circle,fill=black,inner sep=2pt]{} (0,1);
  \draw[wei] (-1, -1) to[out=90,in=-90] (-1,0) to[out=90,in=-90] (-1,1);
  \draw[wei] (0,-1) to[out=90,in=-90](-.5,1);
\end{tikzpicture}\qquad \qquad 
b=
\begin{tikzpicture}[baseline,very thick]
  \draw (-.5,-1) to[out=90,in=-90] node[pos=.8,circle,fill=black,inner sep=2pt]{} (-.5,0) to[out=90,in=-90](1,1);
  \draw (.5,-1) to[out=90,in=-90] (.5,1);
  \draw  (1,-1) to[out=90,in=-90] (-.5 ,1);
  \draw[wei] (0,-1) to[out=90,in=-90] (0,1);
  \draw[wei] (-1, -1) to[out=90,in=-90] (-1,0) to[out=90,in=-90] (-1,1);
\end{tikzpicture}
\]

This algebra ultimately corresponds to the $n$ weight space of the
$\mathfrak{sl}_2$-representation $(\C^2)^{\otimes \ell}$.  The weights that
appear in this representation are $n=\ell,\ell-2,\dots, 2-\ell,-\ell$,
which correspond to $k=0,1,2,\dots,\ell$ black strands.   It is not
obvious, but can be seen from results below (such as Theorem
\ref{th:Gr-iso}) that if $k>\ell$, then the resulting algebra is 0.
Using the connection to Grassmannians we'll describe, this simply
corresponds to the fact that the Grassmannian of $k$
dimensional subspaces of $\C^\ell$ is empty if $k>\ell$.
In particular, $T^\ell$ is a finite dimensional unital algebra.

It will be convenient for us to name several elements of $T^\ell_n$,
which form a generating set:

\begin{equation}
    \tikz[baseline]{
      \node[label=below:{$y_{i;\kappa}$}] at (0,1.2){ 
        \tikz[very thick,xscale=1.4]{
           \draw (-.5,-.5)-- (-.5,.5);
          \draw (.5,-.5)-- (.5,.5) node [midway,fill=black,circle,inner
          sep=2pt]{} ;
          \draw (1.5,-.5)-- (1.5,.5);
          \node at (1,0){$\cdots$};
          \node at (0,0){$\cdots$};
        }
      };
      \node[label=below:{$\psi_{i;\kappa}$}] at (5,1.2){ 
        \tikz[very thick,xscale=1.4]{
          \draw (-.5,-.5)-- (-.5,.5);
          \draw (.1,-.5)-- (.9,.5);
          \draw (.9,-.5)-- (.1,.5);
          \draw (1.5,-.5)-- (1.5,.5);
          \node at (1.1,0){$\cdots$};
          \node at (-.1,0){$\cdots$};
        }
      };
      \node[label=below:{$\iota_{i;\kappa}^+$}] at (0,-1.2){ 
        \tikz[very thick,xscale=1.4]{
          \draw (-.5,-.5)-- (-.5,.5);
          \draw[wei] (.1,-.5)-- (.9,.5);
          \draw (.9,-.5)-- (.1,.5);
          \draw (1.5,-.5)-- (1.5,.5);
          \node at (1.1,0){$\cdots$};
          \node at (-.1,0){$\cdots$};
        }
      };
      \node[label=below:{$\iota_{i;\kappa}^-$}] at (5,-1.2){ 
        \tikz[very thick,xscale=1.4]{
          \draw (-.5,-.5)-- (-.5,.5);
          \draw (.1,-.5)-- (.9,.5);
          \draw[wei] (.9,-.5)-- (.1,.5);
          \draw (1.5,-.5)-- (1.5,.5);
          \node at (1.1,0){$\cdots$};
          \node at (-.1,0){$\cdots$};
        }
      };
    }\label{eq:diagrams}
  \end{equation}
\begin{itemize}
\item Let $e_{\kappa}$ be a diagram with no crossings or dots, and the
  horizontal slice at every value of $y$ corresponding to the function
  $\kappa$. This is an idempotent element of the algebra
  $T^\ell$. Since the function where $\kappa(i)=0$ for all $i\in [1,k]$ is
  especially important, we let $e_0$ denote the sum of the
  idempotents these zero functions for all $k$. 
\item Let $y_{i,\kappa}$ denote the degree $2$ diagram $e_\kappa$ with a single dot
  added on the $i$th strand.
\item Let $\psi_{i,\kappa}$ be the diagram that
  adds a single crossing of the $i$ and $i+1$st strands to $e_\kappa$; if they are
  separated by a red strand, the crossing should occur to the right of
  it. The degree of this element is $-2$ if there is no intervening
  red strand.
\item $\iota^+_{i,\kappa}$ denote the element which creates a single
  crossing between the $i$th black strand of $e_\kappa$ with a red
  strand to its left if this is
  possible without creating black crossings (i.e. if $i-1$ is in the image of $\kappa$). Similarly,
  $\iota^-_{i,\kappa}$ creates crossing with the red strand to the
  right, if this is possible (i.e. if $i$ is in the image of $\kappa$).  These diagrams have degree 1.
\end{itemize}
Note that the diagrams $e_{\kappa},y_{i,\kappa}, $ and $\psi_{i,\kappa}$
have the same sequences at top and bottom; only $\iota^\pm_{i,\kappa}$
change these sequences.  
For a fixed weakly increasing function $\kappa$, we let $\kappa_i^\pm$ be the function
$\kappa_i^\pm(p)=\kappa(p)\pm\delta_{p,p_i^\pm}$ with $p_i^+$ being the
largest integer such that $\kappa(p_i^+)=i-1$, and $p_i^-$ the
smallest integer such that $\kappa(p_i^-)=i$.  If $p_i^\pm$ is not
well-defined since $i-1$ or $i$ is not in the image, then
$\kappa_i^\pm$ is simply not defined.

There is a natural collection of left modules $ T^\ell e_\kappa$ over
the algebra $T^\ell$, given by the idempotents defined above.  These
are projective since they are summands of the left regular module.
In terms of pictures, elements of $P_\kappa=T^\ell e_\kappa $ are
diagrams where we have fixed the strands at the bottom to be the sequence
associated $\kappa$, and where we let the elements of $T^\ell$ act by
attaching them at the top.  

\subsection{A cellular basis}
\label{sec:natural-basis}

When faced with an unfamiliar algebra, one naturally looks for
comforting points of familiarity.  For the algebras we have
introduced, one of these is provided by a basis.   The basis vectors
are indexed by pairs of certain diagrams: each diagram is based on a
Young diagram which fits inside a $k\times (\ell-k)$ box.  We'll draw
partitions in the French style, with the shortest part at the top;
we'll also always give the partition $k$ parts, adding 0's as
necessary, and index these smallest first $\la=(\la_1\leq \la_2\leq
\cdots \leq \la_k)$.  

\begin{definition}
  A {\bf backdrop} for this partition is an association of a number
  between $1$ and $\ell$ at the end to each part of the partition (which we'll write in
  first box of the corresponding row in the Young diagram), even those whose corresponding parts are 0. In
  addition, if we use the same number twice, part of the data of a
  backdrop is to choose an order on the parts with the same number;
  we'll use the notation $i_1,\dots, i_p$ to denote the $p$ instances
  of $i$.  The number of the $j$th row from the top must be $\geq
  j+\la_j$.
\end{definition}
To a backdrop
$\sS$, we have an associated function $\kappa$, where $\kappa(p)$ is
the number of rows with label $<p$.  

Let $\sS$ be a backdrop on a Young diagram; we define an element $B_\sS$ of the
algebra $T^\ell$ as the diagram such that:
\begin{itemize}
\item The bottom of $B_\sS$ has a single black line to the right of the
  $(j+\la_j)$th red line corresponding to the $j$th row (the partition
  condition guarantees that there are no more than one black line
  between red lines; note that this is independent of the labels on
  rows).
\item The top of $B_\sS$ has the number of black strands between the $j$th and
  $(j+1)$st red strands given by the number of rows with label $j$;
  the order on rows with the same label allows us to match up rows
  with black strands at the top.
\item The top and bottom of $B_\sS$ both have black strands labeled by rows of
  the Young diagram; the diagram $B_\sS$ connects the black strands at the top
  and the bottom labeled by the same row.  This diagram isn't unique,
  but we choose one of them with a minimal number of crossings
  arbitrarily; due to the relations
  (\ref{black-bigon}--\ref{red-triple-correction}), any two such
  diagrams will differ by a sum of diagrams with fewer crossings.
\end{itemize}
Note, there are two natural choices of the diagram $B_{\sS}$: {\bf
  left-justified} and {\bf right-justified}.  To construct the
left-justified $B_{\sS}$, as we read from the bottom we add in the
needed crossings of each strand with the strands to its right starting
with the leftmost, and then proceeding to the right; for the
right-justified we start with the rightmost strand and proceed left.
For example, the partition with $(1,1,3,4)$ and $\ell=8,k=4$ with the labels
$(4_2,4_1,7_1,8_1)$ has the associated right-justified diagram $B_{\sS}$ given by 
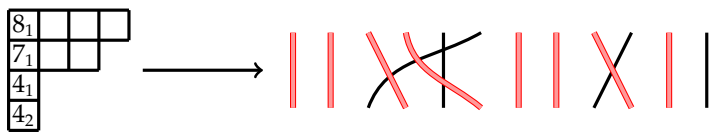
\[\tikz[very thick]{  
\node (a) at (-3,.5){\tikz[very thick]{ 
\draw (0,-.8) -- (0,.8);
\draw (.4,-.8) -- (.4,.8);
\draw (.8,0) -- (.8,.8);
\draw (1.2,0) -- (1.2,.8);
\draw (1.6,.4) -- (1.6,.8);
\draw (0,-.8) -- (.4,-.8);
\draw (0,-.4) -- (.4,-.4);
\draw (0,0) -- (1.2,0);
\draw (0,.4) -- (1.6,.4);
\draw (0,.8) -- (1.6,.8);
\node[scale=.8] at (.2,-.6) {$4_2$};
\node[scale=.8] at (.2,-.2) {$4_1$};
\node[scale=.8] at (.2,.6) {$8_1$};
\node[scale=.8] at (.2,.2) {$7_1$};
}};
\draw[->] (a.0) -- (-.4,.5);
\draw[wei] (0,0) -- (0,1); \draw[wei] (.5,0) --
  (.5,1); \draw (1,0) to[out=70,in=-150] (2.5,1); \draw[wei] (1.5,0) --
  (1,1);\draw (2,0) -- (2,1); \draw[wei] (2.5,0) to[out=140,in=-80]
  (1.5,1); \draw[wei] (3,0) -- (3,1);\draw[wei] (3.5,0) -- (3.5,1);
  \draw (4,0) -- (4.5,1); \draw[wei] (4.5,0) -- (4,1); \draw[wei] (5,0) -- (5,1);
  \draw (5.5,0) -- (5.5,1); }\]
The left-justified diagram for the same backdrop is given by 
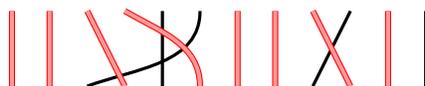
\[\tikz[very thick]{  
\draw[wei] (0,0) -- (0,1); \draw[wei] (.5,0) --
  (.5,1); \draw (1,0) to[out=20,in=-90] (2.5,1); \draw[wei] (1.5,0) --
  (1,1);\draw (2,0) -- (2,1); \draw[wei] (2.5,0) to[out=90,in=-30]
  (1.5,1); \draw[wei] (3,0) -- (3,1);\draw[wei] (3.5,0) -- (3.5,1);
  \draw (4,0) -- (4.5,1); \draw[wei] (4.5,0) -- (4,1); \draw[wei] (5,0) -- (5,1);
  \draw (5.5,0) -- (5.5,1); }\]

Let $B_{\sS}^*$ be the mirror image of the diagram $B_\sS$ through a
horizontal axis.  For two different backdrops $\sS$ and $\sT$ of the
same Young diagram, we have a vector
$C_{\sS,\sT}=B_{\sS}^{\mbox{}}B_{\sT}^*$.  If $\sS$ and $\sT$
are backdrops on different Young diagrams, this product is 0, since
the sequences don't match.  

\begin{theorem}[\mbox{\cite[5.17]{SWschur}}]
  The vectors $C_{\sS,\sT}$ where $\sS$ and $\sT$ range over all
  pairs of backdrops on Young diagrams in a $k\times (\ell-k)$ box form a basis of
  $T^\ell_{\ell-2k}$.  In fact, they are a graded cellular basis of this
  algebra in the sense of Graham and Lehrer \cite{GLcell} and Hu and
  Mathas \cite{HM}.
\end{theorem}
\begin{remark}
  Connecting the combinatorics of \cite[5.17]{SWschur} and of
  backdrops requires some translation.  When applying
  \cite[5.17]{SWschur} to the $\mathfrak{sl}_2$ situtation, we wish to
  consider charged $\ell$-multipartitions which only contain boxes of content 0 in
  their diagram; this is only possible if every component partition is
  a single box or empty, all having charge 0.  Thus, the only
  information is which components in the multipartition are empty, and
  which are non-empty.  The partition $\la$ we consider has one part
  for each component which is a single box, and its length is the
  number of proceeding components which are empty.  

A tableau on such a multipartition is a filling of the boxes; the
numbers we use the backdrop correspond to which alphabet the filling
comes from, and our order corresponds to the order in that alphabet
(that is, our $j_p$ corresponds to $p_j$ in \cite{SWschur}).  Thus,
$(1,1,3,4)$ in our notation corresponds to $(\emptyset, (1),(1),
\emptyset,\emptyset,(1),\emptyset, (1))$, with the tableau having the
entries $2_4,1_4,1_7,1_8$ in that order.  
\end{remark}

A cellular basis of an algebra, amongst other things, supplies a
natural class of modules, the {\bf cell modules} $S_\la$.
\begin{definition}
  The cell module $S_\la$ for a partition $\la$ has a basis given by symbols
  $\{b_{\bS}\}$ for the different backdrops $\mathsf{S}$ on $\lambda$.
  By definition $ C_{\sS',\sT'}b_{\sS}=a_{\sS',\sT',\sS}b_{\sS'}$
  where $a_{\sS',\sT',\sS}$ is the coefficient of $B_{\sS'}$ in the
  basis expansion of $C_{\sS',\sT'}B_{\sS}$.
\end{definition}
For example, if $\ell=5$ and $\la=(1,2)$, then the possible backdrops
are given by choosing $a\in [2,5],b\in [4,5]$, and choosing an order
if $a=b$.  One can easily calculate that there are 10 possibilities:
\[(a,b)\in \{ (2,4),(2,5),(3,4), (3,5), (4_1,4_2),(4_2,4_1),(4,5),(5,4),
(5_1,5_2), (5_2,5_1)\}\] with the associated (right-justified) basis vectors
\begin{multline}
\tikz[baseline,xscale=.45,very thick]{\draw[wei] (0,-.25)--(0,.25);
\draw[wei] (1,-.25)--(1,.25);
\draw[wei] (2,-.25)--(2,.25);
\draw[wei] (3,-.25)--(3,.25);
\draw[wei] (4,-.25)--(4,.25);
\draw (1.5,-.25)--(1.5,.25);
\draw (3.5,-.25)--(3.5,.25);
}\qquad \tikz[baseline,xscale=.45,very thick]{\draw[wei] (0,-.25)--(0,.25);
\draw[wei] (1,-.25)--(1,.25);
\draw[wei] (2,-.25)--(2,.25);
\draw[wei] (3,-.25)--(3,.25);
\draw[wei] (4,-.25)--(4,.25);
\draw (1.5,-.25)--(1.5,.25);
\draw (3.5,-.25)--(4.5,.25);
}\qquad 
\tikz[baseline,xscale=.45,very thick]{\draw[wei] (0,-.25)--(0,.25);
\draw[wei] (1,-.25)--(1,.25);
\draw[wei] (2,-.25)--(2,.25);
\draw[wei] (3,-.25)--(3,.25);
\draw[wei] (4,-.25)--(4,.25);
\draw (1.5,-.25)--(2.5,.25);
\draw (3.5,-.25)--(3.5,.25);
}\qquad 
\tikz[baseline,xscale=.45,very thick]{\draw[wei] (0,-.25)--(0,.25);
\draw[wei] (1,-.25)--(1,.25);
\draw[wei] (2,-.25)--(2,.25);
\draw[wei] (3,-.25)--(3,.25);
\draw[wei] (4,-.25)--(4,.25);
\draw (1.5,-.25)--(2.5,.25);
\draw (3.5,-.25)--(4.5,.25);
}\qquad \tikz[baseline,xscale=.45,very thick]{\draw[wei] (0,-.25)--(0,.25);
\draw[wei] (1,-.25)--(1,.25);
\draw[wei] (2,-.25)--(2,.25);
\draw[wei] (3,-.25)--(3,.25);
\draw[wei] (4,-.25)--(4,.25);
\draw (1.5,-.25) to[out=50,in=-130] (3.25,.25);
\draw (3.5,-.25) -- (3.75,.25);
}\\
\tikz[baseline,xscale=.45,very thick]{\draw[wei] (0,-.25)--(0,.25);
\draw[wei] (1,-.25)--(1,.25);
\draw[wei] (2,-.25)--(2,.25);
\draw[wei] (3,-.25)--(3,.25);
\draw[wei] (4,-.25)--(4,.25);
\draw (1.5,-.25) to[out=50,in=-130] (3.75,.25);
\draw (3.5,-.25) -- (3.25,.25);
}\qquad \tikz[baseline,xscale=.45,very thick]{\draw[wei] (0,-.25)--(0,.25);
\draw[wei] (1,-.25)--(1,.25);
\draw[wei] (2,-.25)--(2,.25);
\draw[wei] (3,-.25)--(3,.25);
\draw[wei] (4,-.25)--(4,.25);
\draw (1.5,-.25) to[out=50,in=-130] (3.5,.25);
\draw (3.5,-.25) -- (4.5,.25);
}\qquad 
\tikz[baseline,xscale=.45,very thick]{\draw[wei] (0,-.25)--(0,.25);
\draw[wei] (1,-.25)--(1,.25);
\draw[wei] (2,-.25)--(2,.25);
\draw[wei] (3,-.25)--(3,.25);
\draw[wei] (4,-.25)--(4,.25);
\draw (1.5,-.25) to[out=30,in=-150] (4.5,.25);
\draw (3.5,-.25) -- (3.5,.25);
}\qquad 
\tikz[baseline,xscale=.45,very thick]{\draw[wei] (0,-.25)--(0,.25);
\draw[wei] (1,-.25)--(1,.25);
\draw[wei] (2,-.25)--(2,.25);
\draw[wei] (3,-.25)--(3,.25);
\draw[wei] (4,-.25)--(4,.25);
\draw (1.5,-.25) to[out=30,in=-150] (4.3,.25);
\draw (3.5,-.25) to[out=40,in=-140] (4.75,.25);
}\qquad \tikz[baseline,xscale=.45,very thick]{\draw[wei] (0,-.25)--(0,.25);
\draw[wei] (1,-.25)--(1,.25);
\draw[wei] (2,-.25)--(2,.25);
\draw[wei] (3,-.25)--(3,.25);
\draw[wei] (4,-.25)--(4,.25);
\draw (1.5,-.25) to[out=30,in=-165] (4.75,.25);
\draw (3.5,-.25) to[out=30,in=-90] (4.3,.25);
}\label{eq:basis}
\end{multline}

We act on these by the usual stacking, applying relations to rewrite our
diagram in the cellular basis, and then discarding all terms involving
basis vectors not on the list above.

If we choose the tautological backdrop $\sT$ where the $j$th row is labeled
with $\la_j+j$, then $C_{\sT,\sT}=e_{\kappa(\la)}$ (as we see in
\eqref{eq:basis} when $(a,b)=(2,4)$).  Since
$e_{\kappa(\la)}^2=e_{\kappa(\la)}$, there's no cellular chain where
$J_i^2\subset J_{i-1}$ for any $i$.  
The theory of cellular bases (in particular, \cite[2.1]{KX}) provides a number of useful corollaries:
\begin{corollary}\label{cellular consequences}\mbox{}
  \begin{enumerate}
\item Every module $S_\la$ has a unique simple quotient $L_\la$, and
  these give a complete, irredundant list of simple modules over
  $T^\ell$.  
  \item The cell modules $S_\la$ are the standard modules of a
    quasi-hereditary structure on the algebra $T^\ell$.  In
    particular, the classes $[S_\la]$ give a free basis for the
    Grothendieck group of finite dimensional $T^\ell$ modules.
\item We write
$\kappa\geq \kappa'$ if this inequality holds pointwise. The module $S_{\la}$ is the quotient of $P_{\kappa(\la)}$ by the
  submodule spanned by all diagrams with a slice that corresponds to
  $\kappa'> \kappa(\la)$.  This is the same the quotient by the submodule spanned by the
  image of every homomorphism $P_\kappa'\to P_{\kappa(\la)}$ for $\kappa'>
  \kappa$.  If there is no $\la$ such that $\kappa= \kappa(\la)$, then the
  corresponding quotient is 0.  
  \end{enumerate}
\end{corollary}

\subsection{An example}
\label{sec:an-example}

The first interesting example is when $\ell=2$ and $k=1$; this
corresponds to the weight $0$ subspace of $\C^2\otimes \C^2$.  

The algebra $T^2_0$ is 5 dimensional:
there are 2 Young diagrams that fit in a $1\times 1$ box,
  corresponding to the partitions $(\emptyset)$ and $(1)$.
  Using the label $1$ or $2$ for $(\emptyset)$ is an acceptable
  backdrop (we call these backdrops $\sT_1,\sT_2$), and for  $(1)$, only
  $2$ is an acceptable label (we call this backdrop $\sS$). 

Thus, we have 5 basis vectors, $C_{\sS,\sS},C_{\sT_2,\sT_1},C_{\sT_2,\sT_2},
C_{\sT_1,\sT_1},C_{\sT_1,\sT_2}$ which
are given by the diagrams:
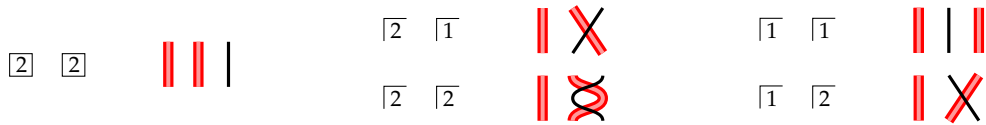
\[ \tikz{ 
\node at (-5,0){ \tikz[xscale=.8,yscale=.6]{\node at (-2,.5){\tikz{\draw(0,.3) -- (0,0) -- (.3,0) --(.3,.3)--cycle;\draw(.7,.3) --
    (.7,0) -- (1,0)--(1,.3)--cycle; \node[scale=.7] at (.16,.16)
    {$2$};\node[scale=.7] at (.86,.16) {$2$};}}; \draw[wei,very thick] (0,0)--(0,1);
    \draw[wei,very thick] (.5,0) --(.5,1);\draw[very thick] (1,0)--(1,1); 
}};
\node at (0,0){ \tikz[xscale=.8,yscale=.6]{
\node at (-2,-1){\tikz{\draw (.3,.3) -- (0,.3) -- (0,0);\draw  (1,.3) --(.7,.3) --
    (.7,0);\node[scale=.7] at (.16,.14)
    {$2$};\node[scale=.7] at (.86,.14) {$1$};}};
\node at (-2,-2.5){\tikz{\draw (.3,.3) -- (0,.3) -- (0,0);\draw  (1,.3) --(.7,.3) --
    (.7,0); \node[scale=.7] at (.16,.14) {$2$};\node[scale=.7] at (.86,.14) {$2$};}};
\draw[wei,very thick] (0,-1.5)--(0,-.5);
    \draw[wei,very thick] (1,-1.5) --(.5,-.5);\draw[very thick] (.5,-1.5)--(1,-.5);
\draw[wei,very thick] (0,-3)--(0,-2);
    \draw[wei,very thick] (.5,-3) to[in=-90,out=90] (1,-2.5) to[in=-90,out=90] (.5,-2);\draw[very thick] (1,-3)
    to[in=-90,out=90] (.5,-2.5) to[in=-90,out=90] (1,-2); }};
\node at (5,0){ \tikz[xscale=.8,yscale=.6]{
\node at (-2,.5){\tikz{\draw (.3,.3) -- (0,.3) -- (0,0);\draw  (1,.3) --(.7,.3) --
    (.7,0); \node[scale=.7] at (.16,.14) {$1$};\node[scale=.7] at (.86,.14) {$1$};}};
\draw[wei,very thick] (0,0)--(0,1);
    \draw[very thick] (.5,0) --(.5,1);\draw[wei,very thick] (1,0)--(1,1); 
\draw[wei,very thick] (0,-1.5)--(0,-.5);
    \draw[wei,very thick]
    (.5,-1.5)--(1,-.5);\draw[very thick] (1,-1.5) --(.5,-.5);
\node at (-2,-1){\tikz{\draw (.3,.3) -- (0,.3) -- (0,0);\draw  (1,.3) --(.7,.3) --
    (.7,0); \node[scale=.7] at (.16,.14) {$1$};\node[scale=.7] at (.86,.14) {$2$};}};
\draw[wei,very thick] (0,0)--(0,1);
}};} \]

In the representation of this algebra defined by Lemma
\ref{faithful-rep}, we let $I_{(a,b)}=I_\kappa$ where $\kappa(1)=1,\kappa(2)=b$.
We have $S=\K[y]$ and $I_{(0,0)}=(y^2), I_{(0,1)}=(y), I_{(1,1)}=S$, so
the space on which they act is $\K[y]/(y^2)\oplus \K$; the algebra
$T^2_0$ is precisely the endomorphisms of this module as a module over
$\K[y]/(y^2)$.    This is that same as Soergel's description
of the principal block of category $\cO$ for $\mathfrak{sl}_2$ using
the Endomorphismensatz \cite{Soe90} as discussed in \cite[\S 5.1.1]{Str03}. 

As noted in \cite[\S 5.1.1]{Str03}, we can also give a description of this algebra as
a quotient of the path algebra of the quiver of a length 2 cycle
\[\tikz[baseline=-3pt]{\node[label=left:$a$] (a) at (0,0) {$\bullet$}; \node[label=right:$b$] (b) at (2,0) {$\bullet$};
  \draw[->] (a.20) -- node[above,midway]{$\psi$} (b.160);\draw[<-] (a.-20) -- node[below,midway]{$\phi$} (b.-160);}\] with $C_{\sS,\sS},
C_{\sT_1,\sT_1}$ giving the length 0 paths at $a$ and $b$, $\psi=C_{\sT_1,\sT_2},\phi=C_{\sT_2,\sT_1}$
giving the length 1 paths, and the single relation $\psi\phi=0$, which follows from 
 \begin{equation*}
  \begin{tikzpicture}[very thick,baseline]
        \draw[wei] (-3.5,-1) -- (-3.5,1);
 \draw (-2.5,-1) .. controls (-1.5,0) ..  (-2.5,1);
       \draw[wei] (-1.5,-1) .. controls (-2.5,0) ..  (-1.5,1);
 \end{tikzpicture}\hspace{5mm}\overset{(\ref{cost})}=\hspace{5mm}\begin{tikzpicture}[very thick,baseline]
     \draw[wei] (.5,-1) -- (.5,1);
    \draw[wei] (2.5,0)  +(0,-1) -- +(0,1);
       \draw (1.5,0)  +(0,-1) -- +(0,1);
       \fill (1.5,0) circle (3pt);
  \end{tikzpicture}\hspace{5mm}\overset{(\ref{cost})}=\hspace{5mm}\begin{tikzpicture}[very thick,baseline]
          \draw[wei] (-2.5,-1) .. controls (-1.5,0) ..  (-2.5,1);
  \draw (-1.5,-1) .. controls (-2.5,0) ..  (-1.5,1);
   \draw[wei] (-.5,-1) -- (-.5,1);
  \end{tikzpicture}\hspace{5mm}\overset{(\ref{violating})}=\hspace{5mm} 0.
\end{equation*}

\subsection{A faithful representation}
\label{sec:faithf-repr}

The relations (\ref{first-nilhecke}--\ref{violating}) may seem strange, but actually, they arise naturally
from a faithful representation.  Fix an integer $k$, and let
$e_p(\mathbf{Y})$ be the elementary symmetric function and
$h_p(\mathbf{Y})$ the complete symmetric function in an alphabet $\mathbf{Y}$.
\begin{definition}
  For each weakly increasing function $\kappa\colon [1,\ell]\to [0,k]$, we define an ideal $I_\kappa\subset
  S=\K[Y_1,\dots, Y_k]$ generated by 
  $h_p(Y_1,\dots, Y_{\kappa(q)})$ for all $q\in [1,\ell]$ and $p>q-\kappa(q)-1$, and
  $h_p(Y_1,\dots, Y_{k})$ for all $p>\ell-k$.
\end{definition}
The most important special case is when $\kappa=0$; in this case, the
$I_\kappa$ is generated by $h_p(Y_1,\dots, Y_{k})$ for $p>\ell-k$.  On
the other hand if $\ell=2,k=1$ and $\kappa(1)=0,\kappa(2)=1$, then we
have that $h_1(Y_1)=Y_1$ is a generator of $I_{\kappa}$ (coming from
$q=1$).  

Taking the coefficients of $t^p$ on LHS and RHS of 
\[\prod_{i=1}^j \frac{1}{(1-tY_i)}=\frac{\prod_{i=j+1}^k
  (1-tY_i)}{\prod_{i=1}^k (1-tY_i)},\] we observe that
\begin{multline}
  h_p(Y_1,\dots,Y_j)=h_p(Y_1,\dots, Y_k)-e_1(Y_{j+1},\dots, Y_k)
  h_{p-1}(Y_1,\dots, Y_k)\\ +e_2(Y_{j+1},\dots, Y_k) h_{p-2}(Y_1,\dots,
  Y_k)-\cdots+(-1)^{k-j}Y_{j+1}\cdots Y_k h_{p-k+j}(Y_1,\dots,
  Y_k)\label{eq:1}
\end{multline}
so $h_p(Y_1,\dots,Y_j)\in I_\kappa$ if $p>\ell-j$. 
Another useful observation is that if
$\kappa(1)>0$, then $1\in I_\kappa$.
\begin{remark}
  As we'll discuss in Section \ref{sec:geom-grassm}, this quotient
  ring is the cohomology ring of a particular smooth Schubert cell in
  its Borel presentation.  Thus, this foreshadows a geometric
  construction of our algebra as discussed in that section. 
\end{remark}
Let $\mathbf{Y}=\{Y_1,\dots, Y_k\}$ with the usual action of the symmetric group $S_k$
and its generators $s_i=(i,i+1)$.
\begin{lemma}\label{faithful-rep}
  The algebra $T^\ell_{\ell-2k}$ acts on the sum $\oplus_\kappa
  S/I_\kappa$ over weakly increasing
  functions sending $e_\kappa$ to the projection to the summand $S/I_\kappa$, and the
  other elements acting by the formulae:
  \begin{align*}
    y_{i,\kappa}(f(\mathbf{Y}))&=Y_i f(\mathbf{Y})\\
    \psi_{i,\kappa}(f(\mathbf{Y}))&=\frac{f(\mathbf{Y})-f(s_i\cdot\mathbf{Y})}{Y_{i+1}-Y_i}&&
    (i\notin \operatorname{im}\kappa)\\
   \iota^+_{i,\kappa}(f(\mathbf{Y}))&= f(\mathbf{Y}) &&
    (i-1\in \operatorname{im}\kappa)\\
    \iota^-_{i,\kappa}(f(\mathbf{Y}))&=Y_i f(\mathbf{Y})  &&
    (i\in \operatorname{im}\kappa)\\
  \end{align*}
\end{lemma}
Since these elements generate the algebra, these formulae determine
the representation.  The formula for general
$\psi_{i,\kappa}$ is more complicated, but easily deduced from the
formulae above.
\begin{proof}
  In \cite[4.12]{Webmerged}, it is shown that these operators on sums of
  copies of the polynomial rings satisfy all the relations of $T^\ell$
  except the violating relation (\ref{violating}), that is, they define an action of the
  algebra $\tilde{T}^\ell$.  

Next, we wish to check that $\tilde{T}^\ell$
  preserves the ideals $I_\kappa$, so that the action on the quotients
  is well-defined.  This is essentially tautological
  for $e_\kappa$ and $y_{i,\kappa}$.  The action of $\psi_i$ commutes
  with multiplication by any polynomial which is symmetric in the variables
  $Y_i$ and $Y_{i+1}$.  Thus, if $i\notin \operatorname{im}\kappa$, we
  have that the defining polynomials for the ideal $I_\kappa$ are
  indeed symmetric in these variables, so this ideal is invariant.

Thus, we have reduced to showing this invariance for $\iota^\pm_{i,\kappa}$.
It's clear that if we have an inequality
  $\kappa'(m)\geq \kappa(m)$ for all $m\in [1,\ell]$, then $I_\kappa\subset I_{\kappa'}$.  Since
  $\kappa^+_i\geq \kappa$, we have that $\iota^+_{i,\kappa}$ induces a map.
For $\iota^-_{i,\kappa}$, we have no such inclusion, but we are not
trying to check that the identity induces a map.  We must instead show
that $Y_ih_p(Y_1,\dots, Y_{\kappa(q)})\in I_{\iota^-_{i,\kappa}}$ for
$p>q-\kappa(q)-1$.

 If $\kappa(q)\neq i$, then this is clear from the
definition.  Now assume $\kappa(q)=i$. As discussed above, if
$\kappa(q+1)\geq i+1$, then we have that $h_p(Y_1,\dots,
Y_{\kappa(q)})$ is already in $ I_{\iota^-_{i,\kappa}}$; the multiplication by $Y_i$ is not even
necessary.  Thus, we need only consider the case where
$\kappa(q+1)=i$. 
In this case, we have the desired inclusion when $p>q-\kappa(q)$, so
we can restrict further to the case $p=q-i$.  
Then we have that  \[Y_ih_{q-i}(Y_1,\dots,
Y_{i})=h_{q-i+1}(Y_1,\dots, Y_{i})-h_{q-i+1}(Y_1,\dots, Y_{i-1}).\] We
have just seen that the former term lies in $I_{\iota^-_{i,\kappa}}$,
and the latter does by definition. 

Finally, it remains to check that this action factors through
$T^\ell$.  As we observed, $S/I_\kappa=0$ if
  $\kappa(1)>0$, so the relation  (\ref{violating}) is immediate
  modulo $I_\kappa$.
\end{proof}

\begin{lemma}\label{T-faithful}
  The action of $T^\ell$ on its polynomial representation is faithful.
\end{lemma}
\begin{proof}
To simplify the exposition here, we'll assume that the result is true
on $e_0T^\ell e_0$; this will be established at the end of  the proof of Proposition \ref{R-T}.

Assume that we have an element $k$ of its kernel.  
Since the kernel is a two-sided ideal, we can multiply at the bottom
and top by elements which sweep all strands to the far
  right, and obtain an element of the kernel $k'$ where both top and bottom
  have $\kappa=0$.  For example:
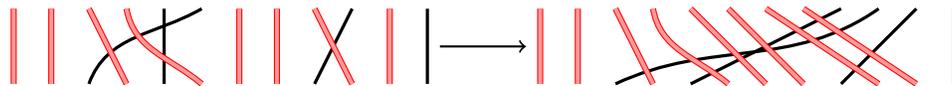
\[\tikz{\node(a) at (0,0) {\tikz[very thick]{  
\draw[wei] (0,0) -- (0,1); \draw[wei] (.5,0) --
  (.5,1); \draw (1,0) to[out=70,in=-150] (2.5,1); \draw[wei] (1.5,0) --
  (1,1);\draw (2,0) -- (2,1); \draw[wei] (2.5,0) to[out=140,in=-80]
  (1.5,1); \draw[wei] (3,0) -- (3,1);\draw[wei] (3.5,0) -- (3.5,1);
  \draw (4,0) -- (4.5,1); \draw[wei] (4.5,0) -- (4,1); \draw[wei] (5,0) -- (5,1);
  \draw (5.5,0) -- (5.5,1); }}; 
\node(b) at (7,0) {\tikz[very thick]{  
\draw[wei] (0,0) -- (0,1); \draw[wei] (.5,0) --
  (.5,1); \draw (1,0) to[out=25,in=-145] (4.5,1); \draw[wei] (1.5,0) --
  (1,1);\draw (2,0) -- (4,1); \draw[wei] (2.5,0) to[out=140,in=-80]
  (1.5,1); \draw[wei] (3,0) -- (2,1);\draw[wei] (3.5,0) -- (2.5,1);
  \draw (4,0) -- (5,1); \draw[wei] (4.5,0) -- (3,1); \draw[wei] (5,0) -- (3.5,1);
  \draw (5.5,0) -- (5.5,1); }}; \draw[thick,->] (a)--(b);}\]

This sweeping operation sends the cellular basis vectors with a given
backdrop to the basis vector where we change every label to $\ell$,
but retaining the order on labels.  In the example above, the labels
change from $(4_2,4_1,7_1,8_1)$ to $(8_2,8_1,8_3,8_4)$.

If we fix the set of labels used in the backdrop, this sends distinct
backdrops to distinct backdrops.  
Similarly, if we fix the slice at the 
top and bottom of the diagram, sweeping sends the basis vectors to a
subset of the basis, which is thus linearly independent.  Thus if $k\neq 0$, then
$k'\neq 0$. 
 
The resulting element
  can be straightened using the relations to be a usual nilHecke
  diagram to the right of all red strands.  This diagram must act
  trivially on $S/I_0$, which is what we obtain
  for the polynomial representation when $\kappa=0$.  Since by
  assumption $e_0 T^\ell e_0$ acts faithfully on this
  space, we must have that all of $T^\ell$ acts faithfully. 
\end{proof}

\subsection{The cyclotomic nilHecke algebra}
\label{sec:cycl-nilh-algebra}

A family of closely related algebras is the cyclotomic nilHecke
algebra $R^\ell=\oplus_n R^\ell_n$, as
discussed in \cite[\S 5.1]{Lauintro}.  The algebra $R^\ell_n$ is the quotient of
the span of Stendhal diagrams with no red strands and $k=(\ell-n)/2$
black strands, with only the relations
(\ref{first-nilhecke}--\ref{black-bigon}) and in place of
(\ref{violating}), we have the relation that $y_1^\ell=0$. Here, we
use $y_i,\psi_i$ to denote diagrams as in  \eqref{eq:diagrams}; since
there are 0 red strands, there is no need to include a function
$\kappa$ (which would have $\emptyset$ as its domain).

\begin{proposition}\label{R-T}
  The map $\imath\colon R^\ell\to T^\ell$ which places a nilHecke diagram to the
  right of $\ell$ red strands induces an isomorphism $R^\ell\cong
  e_0T^\ell e_0$.
\end{proposition}
This is a special case of \cite[4.21]{Webmerged}.
\begin{proof}
  First we check that this map is well-defined. The relations (\ref{first-nilhecke}--\ref{black-bigon})
  are unchanged and thus hold. We need only check that the
  image $y_1e_0$ of $y_1$ under this homomorphism satisfies $y_1^\ell e_0=0$.

  This is an immediate consequence of the relations (\ref{red-dot}--\ref{violating}):
 \begin{equation*}
  \begin{tikzpicture}[very thick,baseline=-.6cm]
    \draw[wei] (1.2,-1) .. controls (2.8,0) ..  (1.2,1);
\node at (2.9,0){$\cdots$};
\draw[wei] (2.2,-1) .. controls (3.8,0) ..  (2.2,1);
       \draw (2.8,-1) .. controls (1.2,0) ..  (2.8,1);
           \node at (.3,0) {=};
    \draw (-1.2,0)  +(0,-1) -- +(0,1);
       \draw[wei] (-2.8,0)  +(0,-1) -- +(0,1);
\node at (-2.3,0){$\cdots$};
       \draw[wei] (-1.8,0)  +(0,-1) -- +(0,1);
       \fill (-1.2,0) node[right]{$\ell$} circle (3pt);
\node at (4.5,0){$=0.$};
  \end{tikzpicture}
\end{equation*}

Consider an element in the image of $\imath$; this is obtained by
starting with the idempotent $e_0$, and multiplying it by elements
$\psi_i$ and $y_i$. 
  From the  formulae of Lemma \ref{faithful-rep}, we see that the action of
  these elements are given by Demazure operators and multiplication on
  $S/I_0$.
Thus, the usual action on the nilHecke algebra
  on polynomials, as in \cite[\S 2.3]{KLI}, factors through the map
  $\imath$.  
That is, we have a commutative diagram:
\begin{equation}
\tikz[->,thick,baseline]{
\matrix[row sep=10mm,column sep=20mm,ampersand replacement=\&]{
\node (a) {$R^\ell_n$}; \& \& \node (c) {$\End(S/I_0)$}; \\
\& \node (e) {$e_0 T^\ell_ne_0$}; \& \\
};
\draw (a) -- (c) ;
\draw (a) -- node[below left]{$\imath$} (e) ; 
\draw (e) -- (c) ; 
}\label{eq:triangle}
\end{equation}
 In \cite[5.3]{Lauintro}, Lauda shows that this
  action of $R^\ell$ induces an isomorphism between the cyclotomic nilHecke
  algebra and a matrix ring over the cohomology of the Grassmannian,
  so the top arrow of \eqref{eq:triangle} is injective. Thus $\imath$
  must be injective as well.

In order to see that $\imath$ is surjective as well, we must show
  that any diagram with $\kappa(i)=0$ for all $i$ at both top and bottom can be
  written as a sum of diagrams where all black strands stay right of
  all red ones.  This is easily achieved using the relations
  (\ref{red-triple-correction}) and (\ref{cost}).

Note that the fact that we have an isomorphism $R^\ell\cong e_0T^\ell
e_0$ and the fact that the top arrow of \eqref{eq:triangle} is
injective shows that the action map $e_0 T^\ell_ne_0\to \End(S/I_0)$
is injective, as needed in the proof of Lemma \ref{T-faithful}.
\end{proof}

\subsection{Decategorification}
\label{sec:decategorification}

The algebra $T^\ell$ appears in a number of different ways.  Perhaps most
significant for us is that it categorifies certain tensor product
representations of $\mathfrak{sl}_2$.  

\begin{definition}
  We let $T^\ell_n\mmod$ be the category of finitely generated
  left $T^\ell_n$ modules, and $T^\ell_n\gmod$ the category of
  finitely generated graded
  modules over the same algebra.
\end{definition}

We have a natural map
$\phi\colon T^\ell_n\to T^\ell_{n-2}$ given by adding a black strand
at far right.  This map is a homomorphism but not unital; instead it sends
the identity to an idempotent $e_\phi$ given by the sum of the
idempotents $e_\kappa$ where the rightmost strand is black, i.e. $\kappa(\ell)<k$.
\begin{definition}
  We let \[\fF(M)=T^\ell_{n-2}\otimes_{T^\ell_n}M\colon T^\ell_n\mmod \to T^\ell_{n-2}\mmod\] be the induction
  functor for this map.

We let $\fE(M)=e_\phi M$ be the functor  biadjoint (up to grading
shift) with $\fF$.
\end{definition}
The functor $\fE$ is by definition the right adjoint of $\fF$.  The
fact that is is left adjoint is not obvious; it follows from the
existence of a categorical $\mathfrak{sl}_2$-action defined by these functors:
\begin{theorem}[\mbox{\cite[4.28]{Webmerged} }]
  The functors $\fE$ and $\fF$ define a categorical action
  of $\mathfrak{sl}_2$,
  in the sense of Chuang and Rouquier \cite{CR04}.  
\end{theorem}

Similarly, we have a nonunital inclusion $\eta \colon T^\ell_n\to
T^{\ell+1}_{n+1}$, by simply adding a new red strand at the far right,
and we let $\fI$ be the extension of scalars functor $\fI(M)=T^{\ell+1}_{n+1}
\otimes_{ T^\ell_n}M$ for the map $\eta$.   

Note that our projective modules $P_\kappa$ can also be built
with the functors $\fF$ and $\fI$ as follows: if we use $P_\emptyset$
to denote the unique irreducible module over $T^{0}_{0}\cong \K$, then
\begin{equation}
P_{\kappa}\cong \fF^{k-\kappa(\ell)}\fI
\fF^{\kappa(\ell)-\kappa(\ell-1)}\fI\cdots\fI\fF^{\kappa(1)}P_\emptyset\label{eq:2}
\end{equation}
since the RHS is defined as induction by an algebra inclusion $\K\to
T^\ell$ sending $1\mapsto e_{\kappa}$, and $P_\kappa=T^\ell e_{\kappa}\cong
T^\ell\otimes_{\K e_{\kappa}}\K$.

Now, we'll relate this picture to the tensor product $(\C^2)^{\otimes
  \ell}$; for notational reasons, it will be easier to think of this
as a $\ell+1$-term tensor product with a trivial module spanned by
$\mathbbm{1}$ as the first term. 
 We'll always consider $\C^2$ with its usual basis $\left\{\left[\begin{smallmatrix}
  1\\0
\end{smallmatrix}\right],\left[\begin{smallmatrix}
  0\\1
\end{smallmatrix}\right]\right\}$, and the tensor product $(\C^2)^{\otimes
  \ell}$ with the induced tensor product basis.  We can label these
basis vectors as $s_\la$ where $\lambda$ is the partition which has a part of each time $\left[\begin{smallmatrix}
  0\\1
\end{smallmatrix}\right]$ appears, with the length of the part being
the number of $\left[\begin{smallmatrix}
  1\\0
\end{smallmatrix}\right]$'s to the left of that instance.  For a
vector of weight $n$, the number
of parts is $k=(\ell-n)/2$ (possibly including parts of length 0), and
the resulting partition fits inside a $k\times (\ell-k)$ box. For
example, the basis vectors \[  \left[\begin{smallmatrix}
  1\\0
\end{smallmatrix}\right]\otimes \left[\begin{smallmatrix}
  1\\0
\end{smallmatrix}\right], \left[\begin{smallmatrix}
  0\\1
\end{smallmatrix}\right]\otimes \left[\begin{smallmatrix}
  1\\0
\end{smallmatrix}\right],\left[\begin{smallmatrix}
  1\\0
\end{smallmatrix}\right]\otimes \left[\begin{smallmatrix}
  0\\1
\end{smallmatrix}\right],\left[\begin{smallmatrix}
  0\\1
\end{smallmatrix}\right]\otimes \left[\begin{smallmatrix}
0 \\1
\end{smallmatrix}\right]\] correspond to the partitions
$\emptyset,(0), (1), (0,0)$ (which are the only partitions fitting in
a $0\times 2,1\times 1,$ or $2\times 0$ box).  Note that the basis
vector is only uniquely specified if $\ell$ and $\la$ are fixed. 

This process is often visualized by drawing a path which
travels SW to NE with $\left[\begin{smallmatrix}
  0\\1
\end{smallmatrix}\right]$ corresponding to a vertical line segment and $\left[\begin{smallmatrix}
  1\\0
\end{smallmatrix}\right]$ to a horizontal.  This will connect the SW
and NE corners of a $k\times (\ell-k)$.  The region NW of this line
inside the box 
is the Young diagram of the partition in French notation.  Below, we
show the examples of $\emptyset,(0), (1), (0,0)$ with $\ell=2$ and
$(1,2)$ with $\ell=5$; in the last case, we have $s_{(1,2)}=\left[\begin{smallmatrix}
  1\\0
\end{smallmatrix}\right]\otimes \left[\begin{smallmatrix}
  0\\1
\end{smallmatrix}\right]\otimes \left[\begin{smallmatrix}
  1\\0
\end{smallmatrix}\right]\otimes \left[\begin{smallmatrix}
  0\\1
\end{smallmatrix}\right]\otimes \left[\begin{smallmatrix}
  1\\0
\end{smallmatrix}\right]$:
\begin{equation}
 \tikz{\draw[very thick, baseline] (-.5,0) -- (.5,0) ;}\qquad
\tikz{\draw[very thick, baseline] (-.25,-.25) -- (-.25,.25)--
  (.25,.25) ;\draw[thin, dashed] (-.25,-.25) -- (.25,-.25)-- (.25,.25)
  ;} \qquad
\tikz{\fill[pattern=north west lines ] (-.25,-.25) -- (.25,-.25)--
  (.25,.25) --(-.25,.25) --cycle; \draw[very thick, baseline] (-.25,-.25) -- (.25,-.25)--
  (.25,.25) ;\draw[thin, dashed] (-.25,-.25) -- (-.25,.25)-- (.25,.25)
  ;  }\qquad \tikz{\draw[very thick, baseline] (0,-.5) -- (0,.5) ;}\qquad
\tikz{\fill[pattern=north west lines ] (-.5,-.5) -- (-.5,.5)--
  (.5,.5) --(.5,0) --(0,0) --(0,-.5) --cycle; \draw[very thick, baseline] (-.5,-.5) -- (0,-.5)--
 (0,0) -- (.5,0)-- (.5,.5) -- (1,.5);\draw[thin, dashed] (-.5,-.5) --
 (-.5,.5)-- (1,.5)--(1,-.5) --cycle;\draw[thin, dashed] (.5,-.5) --
 (.5,.5)-- (0,.5)--(0,-.5) --cycle; \draw[thin, dashed] (-.5,0) --
 (1,0);  }\label{eq:partitions}
\end{equation}

Let $K^0(T^\ell_n\mmod)$ be the Grothendieck group of the category
$T^\ell_n\mmod$.  
\begin{theorem}[\mbox{\cite[4.38]{Webmerged}}]
\label{th:Gr-iso}
  The  sum $\bigoplus_n K^0(T^\ell_n)\otimes_\Z\C$ is canonically isomorphic to
  $(\C^2)^{\otimes \ell}$, via the map sending $[S_\lambda]\mapsto s_\la$.  This isomorphism sends $
  K^0(T^\ell_n)\otimes_\Z\C$ to the weight $n$ subspace. 
\end{theorem}
\begin{proof}
  The classes $[S_\lambda]$ are a basis of $\bigoplus_n
  K^0(T^\ell_n) \otimes_\Z\C$ by Corollary \ref{cellular
    consequences}, and $s_\la$ are a basis of $(\C^2)^{\otimes \ell}$
  by standard results about tensor products.  Thus we have an
  isomorphism of vector spaces.
\end{proof}

Let $I\colon (\C^2)^{\otimes \ell}\to (\C^2)^{\otimes \ell+1}$ be the
inclusion $v\mapsto v\otimes
\left[\begin{smallmatrix}
  1\\0
\end{smallmatrix}\right]$ given by tensor product with the obvious highest
weight vector $\left[\begin{smallmatrix}
  1\\0
\end{smallmatrix}\right]\in \C^2$.   This map sends basis vectors to
basis vectors, and leaves the resulting partition unchanged.  Let $E,F$ denote the
standard Chevalley generators of $\mathfrak{sl}_2$, acting as usual on
the tensor product representation.  That is, they act by the sums:
\[E=\sum_{k=1}^\ell 1^{\otimes  k-1}\otimes \left[\begin{smallmatrix}
  0&1\\0&0
\end{smallmatrix}\right]\otimes 1^{\otimes \ell-k}\qquad F=\sum_{k=1}^\ell 1^{\otimes  k-1}\otimes \left[\begin{smallmatrix}
  0&0\\1&0
\end{smallmatrix}\right]\otimes 1^{\otimes \ell-k}.\]
One can easily work out the action of these on the vectors $s_\la$.
The vector $Es_\la$ is a sum of the $s_{\mu}$'s obtained by deleting
the $i$th part $\la_i$ from $\la$, and increasing all parts $\la_j$
for $j>i$ by
1. We let $\chi^+_\la$ be the set of such partitions. For example, $Es_{(0,0)}=s_{(1)}+s_{(0)}$ with the two terms
coming from deleting the first and second parts respectively, and
$Es_{(1,2)}=s_{(1)}+s_{(3)}$.   If we
draw the partition inside its box as in \eqref{eq:partitions},
$\chi^+_\la$ is the set of all partitions obtained by turning one vertical segment of the boundary horizontal.
The operator $F$ acts in the same way on the transpose partition, that
is, it sums over all ways of turning one horizontal segment vertical;
we let $\chi^-_\la$ be the set of all partitions obtained this way.
Note that $\chi^-_\la$ depends on $\ell$: for example,
$\chi^-_{\emptyset}=\{ (0),(1),\dots, (\ell-1)\}$.  In contrast, $\chi^+_\la$ does
not depend on $\ell$.  

In order to show that this isomorphism is equivariant, let us consider
how $\fE,\fF,\fI$ act on standard modules.
\begin{proposition}\label{prop:act-filter}\mbox{}
  \begin{enumerate}
  \item   The module $\fE S_\lambda$ has a filtration by the standard modules
  $S_{\mu}$ for $\mu \in  \chi^+_\la$.
\item   The module $\fF S_\lambda$ has a filtration by the standard modules
  $S_{\mu}$ for $\mu \in  \chi^-_\la$.
\item For all $\la$, we have  $\fI S_\lambda\cong S_\lambda$.
\end{enumerate}
\end{proposition}
\begin{proof}
Throughout this proof, we work with the left-justified basis.

First, we prove (1).  The restricted module $e_\phi S_\lambda$ is spanned by the basis
  vectors for backdrops in which $\ell$ appears at least once.  The
  submodules $M_i$ of the filtration are the span of the vectors where
  the largest occurence of $\ell$ (that is, $\ell_p$ where $p$ is the
  number of rows with label $\ell$) appears in row $j$ with $j\geq
  i$.  This is the same as looking at the black strand which is at the
  far right of the diagram at the top, and requiring that it be in the
  rightmost $k-i+1$ strands at the bottom of the diagram.  The action
  of $T^\ell_{n+2}$  can only change which strand connects to the far
  right terminal at the top if a dot slides across that strand using
  (\ref{first-nilhecke}).  In that case, the resulting diagrams will
  still lie in $M_i$: the terminal at the bottom can move leftward,
  but not rightward.  

Let $\mu_i$ be the partition obtained by removing the $i$th smallest
part from the partition $\la$, that is, flipping the $i$th vertical
segment on the boundary when reading from the SE.
There is a surjective map from $S_{\mu_i}\to M_i/M_{i+1}$, sending the
basis vector for a backdrop $\sS$ on $\mu_i$ to the vector for the 
backdrop $\bS'$ on $\la$, with the $i$th part given label $\ell$, larger than
any other $\ell$ which appears, and all other labels the same as $\bS$.    This is shown in the diagram below:
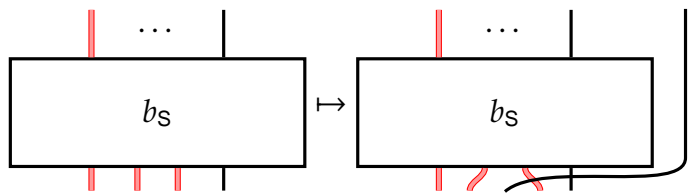
\[\tikz[baseline,very thick,xscale=2]{ \node (a) [inner xsep=50pt, inner ysep=15pt,draw] at (0,0){$b_{\sS}$};
\draw[wei] (a.140) -- +(0,.63);
\draw (a.40) -- +(0,.63);
\node at (0,1.05){$\cdots$}; 
\draw[wei] (a.-110) -- +(0,-.3);
\draw[wei] (a.-70) -- +(0,-.3);
\draw[wei] (a.-140) -- +(0,-.3);
\draw (a.-40) -- +(0,-.3);
        }\mapsto \tikz[baseline,very thick,xscale=2]{ \node (a) [inner xsep=50pt, inner ysep=15pt,draw] at (0,0){$b_{\sS}$};
\draw (a.40) -- +(0,.63);
\draw[wei] (a.140) -- +(0,.63);
\node at (0,1.05){$\cdots$}; 
\draw[wei] (a.-140) -- +(0,-.3);
\draw[wei] (a.-110) to[out=-90,in=90] +(-.1,-.3);
\draw[wei] (a.-70) to[out=-90,in=90] +(.1,-.3);
\draw (a.-40) -- ++(0,-.3);
\draw (0,-1.05) to[in=-90,out=60] (1.2,-.6)
to[in=-90,out=90] (1.2,1.38);
        }\]
Each one of these maps must be injective, since the dimension of $\fE
S_\lambda$ is the same as the sum of the dimensions of $S_{\mu}$.

Next, we turn to (2): the module $S_\lambda$ is a quotient of
$P_{\kappa(\lambda)}$, by the submodule generated $B_{\sS}^*$ for
$\sS$ not the tautological backdrop on $\la$, and by the exactness of
$\fF$, we also have a surjective map $\fF P_{\kappa(\lambda)}\to \fF S_\lambda$.
For each $\mu\in \chi^-_\la$, there is a special backdrop $\sS_\mu$ with the largest value of
$\ell$ on the ``new'' part and all other parts with the same
labeling from the tautological backdrop on $\la$. 
Let $K$ be the span of
the vectors $C_{\sS,\sS'}$ with $\sS$ any backdrop other
than $\sS_\mu$ for $\mu \in \chi^-_\la$. it's easy to see that $K$
lies in the kernel of the map $\fF P_{\kappa(\lambda)}\to \fF S_\lambda$, the cellular basis
structure shows that $K$ is a submodule, and the generating
vectors of the kernel lie in $K$, so it follows that $K$ is precisely the
kernel.

This shows that $C_{\sS_\mu,\sS'}$ for $\mu \in \chi^-_\la$ are a
basis of $\fF S_\lambda$.  If we let $N_i$ be the span of these
vectors where the ``new'' part is the $j$th, for $j\leq i$, then we can
see that $N_i$ is a submodule, by the cellular structure.  We have an
isomorphism $S_{\mu_i} \cong N_i/N_{i-1} $, sending $b_{\sS'}\mapsto
C_{\sS_\mu,\sS'}$.  

Finally, we wish to prove (3).  
In this case, both modules are
quotients of the projective $P_{\kappa'}$ where $\kappa'$ is the
extension of $\kappa(\la)$ to $[1,\ell+1]$ by $\kappa'(\ell+1)=k$.  By the description of
Corollary \ref{cellular consequences}(3), $S_{\lambda}$ is the
quotient by all maps from $P_{\kappa''}$ with $\kappa''>\kappa'$.  Any
such map can be assumed to be in the image of $\fI$, since we must
have $\kappa(\ell+1)=k$ as well, and all cellular basis vectors with
bottom $\kappa'$ and top $\kappa''$ correspond to backdrops that don't
use $\ell+1$ as a label; in this case the basis vector is obtained by
adding a red strand at the far right to the basis for the same
backdrops considered for $\ell$ red strands.  This shows that the same
submodule is killed by the map $P_{\kappa'}\cong \fI
P_{\kappa(\la)}\to \fI S_\la,$ so we have the desired isomorphism.
\end{proof}
The functors $\fE,\fF$ and $\fI$ are exact, and thus naturally induce
maps $[\fE],[\fF]$ and $[\fI]$ on the Grothendieck group.  
\begin{corollary}
  The isomorphism of Theorem \ref{th:Gr-iso} intertwining the 
 induced maps
  $[\fE],[\fF],[\fI]$ on the Grothendieck group with the actions of $E,F,I$
  on $(\C^2)^{\otimes \ell}$.  
\end{corollary}
 
It immediately follows from this theorem and \eqref{eq:2} that we can
describe the image of $[P_\kappa]$ under this map: reading from left
to right, each time we encounter a black strand, we apply $F$ and each
time we encounter a red one, we apply $I$.  That is:
\begin{equation}
[P_\kappa]\mapsto p_\kappa:= F^{k-\kappa(\ell)}
(F^{\kappa(\ell)-\kappa(\ell-1)}\cdots F^{\kappa(2)-\kappa(1)}( F^{\kappa(1)}\mathbbm{1}\otimes
\left[\begin{smallmatrix}
  1\\0
\end{smallmatrix}\right])\cdots \otimes
\left[\begin{smallmatrix}
  1\\0
\end{smallmatrix}\right]). \label{eq:3}
\end{equation}

\begin{remark}
  The decategorification results of this section can be ``upgraded''
  to take into account the grading on $T^\ell$.  If we consider the
  abelian category $T^\ell\gmod$, then the
  Grothendieck group of this category
  is naturally a $\Z[q,q^{-1}]$-module where $q$ acts by decreasing the
  grading of the module.  That is, $q[M]=[M\{1\}]$, where $M\{1\}$
  denotes an isomorphic module with the grading decreased by $1$.  The action induced by $\fE$ and $\fF$ on
  this category doesn't satisfy the relations of $\mathfrak{sl}_2$,
  but rather of the quantum group $U_q(\mathfrak{sl}_2)$.  See
  \cite{KLIII,Webmerged} for more details.
\end{remark}

\section{The geometry of Grassmannians}
\label{sec:geom-grassm}

{\small \it 
\begin{quotation}
In these days the angel of topology and the devil of abstract algebra
fight for the soul of each individual mathematical domain. 

\hfill Hermann Weyl (1939)
\end{quotation}
}

In this section, we will give a geometric description of the algebra
$T^\ell$: we will realize it as a convolution algebra in homology for
some natural correspondences over Grassmannians.  This construction
fits in with many geometric constructions of KLR type algebras, from
\cite{SWschur,WebwKLR,VV} and others. Of course, before
doing this, we need to give a bit of background on the geometry of Grassmannians.

\subsection{Definitions}
\label{sec:definitions}

Fix integers $k,\ell$ and let $\Gr(k,\ell)$ be the Grassmannian of
$k$-planes in $\C^\ell$.  Let $n=\ell-2k$.
This projective variety has a well-known decomposition into Schubert
cells.  We have a fixed flag $\C^1\subset \C^2\subset
\cdots\subset\C^{\ell-1}\subset \C^\ell$ with $\C^m$ identified with
the span of the first $m$ unit vectors.  For each weakly increasing
function $\kappa\colon [1,\ell]\to [0,k]$ such that
$k-\ell+m\leq\kappa(m)< m$ and $\kappa(m+1)\leq \kappa(m)+1$, we let \[X_{\kappa}=\{V\in
\Gr(k,\ell)\mid\dim(V\cap \C^{m-1})=\kappa(m)\}\] and also consider its closure, the Schubert
variety
\[\bar{X}_\kappa=\{V\in
\Gr(k,\ell)\mid\dim(V\cap \C^{m-1})\geq \kappa(m)\}.\]
The functions $\kappa$ satisfying the conditions we have written are precisely
those of the form $\kappa(\la)$ for some partition.  We can
reconstruct $\la$ from $\kappa$ via the formula $\la_p=\max\{m|\kappa(m)< p\}-p$.

Geometrically, if we consider the
graph of $\kappa$ in a $k\times \ell$ rectangle, then $\kappa$ must
remain inside a lozenge, and the partition is the size of the rows in
the lozenge above the graph.  For example, if $\ell=7$ and $k=3$, and $\kappa$ applied to
$1,\dots, 7$ gives $0,0,1,1,2,3,3$, then this graph will look like:
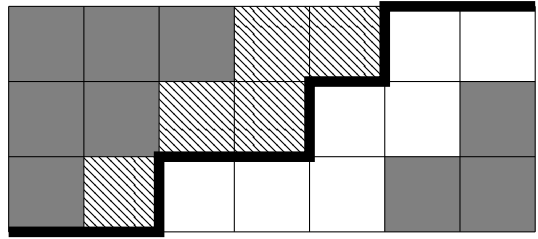
\begin{equation*}
  \begin{tikzpicture}
\fill[pattern=north west lines ] (1,0) -- (1,1) -- (2,1) --
(2,2) -- (3,2) -- (3,3) --(5,3) -- (5,2) -- (4,2) -- (4,1) -- (2,1)  -- (2,0)  --(0,0)
-- cycle;
\fill[gray] (0,0) -- (1,0) -- (1,1) -- (2,1) -- (2,2) -- (3,2) -- (3,3)--
(0,3) --cycle;
\fill[gray] (5,0) -- (5,1) -- (6,1) -- (6,2) -- (7,2) -- (7,0) --cycle;
    \draw (0,0) -- (0,3);
    \draw (1,0) -- (1,3);
    \draw (2,0) -- (2,3);
    \draw (3,0) -- (3,3);
     \draw (4,0) -- (4,3);
    \draw (5,0) -- (5,3);
    \draw (6,0) -- (6,3);
    \draw (7,0) -- (7,3);
\draw (0,0) -- (7,0);
\draw (0,1) -- (7,1);
\draw (0,2) -- (7,2);
\draw (0,3) -- (7,3); 
\draw[line width=4pt] (0,0) -- (2,0) -- (2,1) -- (4,1) -- (4,2) -- (5,2)
--(5,3) -- (7,3);
\end{tikzpicture}
\end{equation*}
The gray regions denote where the graph of $\kappa$ is forbidden and
the rows of the hatched region give the desired Young
diagram in a box (with French notation).  In this example, we
obtain:
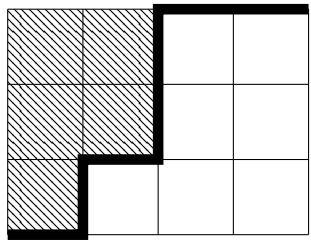
\[ \begin{tikzpicture}
\fill[pattern=north west lines ] (0,0) -- (1,0) -- (1,1) -- (2,1) --
(2,3) -- (4,3) -- (0,3)
-- cycle;
    \draw (0,0) -- (0,3);
    \draw (1,0) -- (1,3);
    \draw (2,0) -- (2,3);
    \draw (3,0) -- (3,3);
     \draw (4,0) -- (4,3);
\draw (0,0) -- (4,0);
\draw (0,1) -- (4,1);
\draw (0,2) -- (4,2);
\draw (0,3) -- (4,3); 
\draw[line width=4pt] (0,0) -- (1,0) -- (1,1) -- (2,1) -- (2,3) -- (4,3);
\end{tikzpicture}\]

Each Schubert variety has a resolution of
singularities of the form \[\tilde{X}_\kappa=\{V_0=\{0\}\subset V_1
\subset \cdots \subset V_{\ell-1}\subset 
V_\ell \mid V_m\subset \C^{m-1},\dim V_m=\kappa(m)\}.\]  This has a natural map
$\tilde{X}_\kappa\to \bar{X}_\kappa$ forgetting all entries of the
flag except for $V_\ell$.  This map is a resolution of singularities
since $\tilde{X}_\kappa$ is  smooth (since it is a tower of
Grassmannian fibrations), and it is an isomorphism over the locus
$X_\kappa$ (since we are forced to take $V_m=V\cap \C^{m-1}$).

Now, let me introduce a closely related collection of varieties whose
import will not be immediately clear.  We introduce a fibration
$p_\kappa\colon Y_\kappa\to \tilde X_\kappa$ where we choose a complete flag on
$V_i/V_{i-1}$.  That is, 
\[Y_\kappa=\{W_0=\{0\}\subset W_1\subset \cdots \subset W_{k-1}\subset
W_k \mid W_{\kappa(m)}\subset \C^{m-1},\dim W_m=m.\}\] 
Note that this space makes sense for any weakly increasing $\kappa$, even
if it doesn't meet the inequalities to match a Schubert variety, and 
that $Y_\kappa$ is actually a smooth Schubert variety in the full
flag variety for any $\kappa$.  

\subsection{Convolution}

We now want to use this geometry to define an algebra, using the
method of convolution in homology.  This method is discussed in much
greater detail in \cite[\S 2.7]{CG97}.  Whenever we have an algebraic map
between smooth projective varieties $Y\to X$, the homology of the fiber
product $A=H_*(Y\times_XY;\K)$ inherits a product structure. 
Consider the projections  $p_{12},p_{13}, p_{23}\colon Y\times_X
Y\times_XY\to Y\times_X Y$ which forget the third, second and first
terms respectively.  
The product is defined on $a,b\in A$ by 
\[a\star b=(p_{13})_*(p_{12}^*a\cap p_{ 23}^* b)\]
using the fact that for maps between smooth compact manifolds, there
are pullback {\it and} pushforward maps in homology.

For the reader unfamiliar with this technique, we'll only need to
directly apply the definition for a few calculations.  First note that
pushforward by the diagonal map on the homology of $Y$ induces an
inclusion of algebras $\Delta_*\colon H_*(Y;\K)\to A$; the product
structure on homology is the intersection product.  More general elements
can be induced by a space $Z$ with two maps $h_1,h_2\colon Z\to Y$, such that both
induce the same map $Z\to X$; in this case, we consider the
pushforward $(h_1\times h_2)_*[Z]\in H_*(Y\times_XY)$.

We'll let   $X=\Gr(k,\ell)$ and $Y=\bigsqcup_\kappa Y_\kappa$ with
$p\colon Y\to X$ the usual projection. As before, we
define $n$ by $\ell-n=2k$, and denote the resulting convolution algebra
by $A^\ell_n$.  

We'll
abuse notation, and let $W_{m}/W_{m-1}$ denote the line bundle  on $Y$
whose fiber at each point is given by this line, and let
$e(W_{m}/W_{m-1})$ be the homology class given by the divisor of this
line bundle, that is, the Poincar\'e dual of its Euler class.

Let \[Z_{i,\kappa,\kappa'}=\{(W_{\bullet},W'_{\bullet})\in Y_{\kappa}\times
Y_{\kappa'}\mid W_{j}= W_j'\text{ for all } j\neq i\};\] this variety is
endowed with maps $h_1,h_2\colon Z_i\to Y$ forgetting the second and
first entry of the pair respectively. We'll also use $Z_0$ to denote
the space where we require the flags to be equal.  

There has been a profusion of variations on KLR algebras in recent
years.  These algebras, in most cases, can be geometrically realized
as a convolution algebra.  Examples include \cite[3.6]{VV}
\cite[Th. B]{WebwKLR} and \cite[0.1]{SauQH}.  In this context, let us
state the relevant theorem, those we will need to develop a few lemmas
before completing its proof:
\begin{theorem}\label{main}
  The algebras $T^\ell_n$ and $A^\ell_n$ are isomorphic via the map
  \begin{align*}
    e_\kappa&\mapsto \Delta_*[Y_\kappa]& y_m&\mapsto e(W_{m}/W_{m-1})\\
    \psi_{i,\kappa}&\mapsto (h_1\times h_2)_*
    [Z_{i,\kappa,\kappa}] &\iota_{i,\kappa}^{\pm}&\mapsto \pm (h_1\times h_2)_*[Z_{0,\kappa^{\pm}_i,\kappa}]
  \end{align*}
\end{theorem}  
How is one to think about this theorem? While I would argue that this
is really the correct definition of $T^\ell_n$, and that one should
then derive the diagrammatic description, this is just moving the
problem around.  The key property of $A^\ell_n$ is that it acts on the
homology $H_*(Y;\K)$.  Let $x_i=e(W_{m}/W_{m-1})$.    

It might seem daunting to analyze such a convolution algebra, but it
can be done relatively easily using a natural representation arising
from its definition.  Examples include \cite[\S 2]{VV} and \cite[2.6]{SWschur}.
\begin{lemma}\label{A-faithful}
  The action of $A^\ell_n$ on $H_*(Y;\K)$ is faithful.
\end{lemma}
\begin{proof}
As shown in Ginzburg and Chriss \cite[8.6.7]{CG97}, the algebra $A^\ell_n$ is
the self-Ext algebra of $p_*\K_Y$, where $\K_Y$ is the sheaf of
locally constant $\K$-valued functions on $Y$, so it suffices to show that any
Ext between summands of this pushforward induces a non-zero map on hypercohomology.
The pushforward $p_*\K_Y$ is a parity sheaf by \cite[4.8]{JMW}.  Faithfulness follows from the same argument as
  \cite[3.2.6]{Soe00} (note that the argument in the paper is
  incorrect, and corrected in \cite{Soecor}). 
\end{proof}

\begin{lemma}\label{GR}
  The homology $H_*(Y_\kappa,\K)$ is isomorphic as an algebra under
  intersection product to $\K[x_1,\dots, x_k]$ modulo the ideal
  $I_\kappa$ generated by 
  $h_p(x_{1},\dots, x_{\kappa(q)})$ if $p> q-\kappa(q)-1$.
\end{lemma}
\begin{proof}
The space $Y_\kappa$ is a Schubert cell in a partial flag variety,
defined requiring ``non-crossing'' inclusions (since it involves no
conditions of the form $\C^k\subset V_p$, only of the opposite form);
in particular, this Schubert variety is smooth. 

From the main theorem of \cite{GR02}, the homology of this smooth
Schubert variety is generated by $x_{1},\dots, x_{k}$, only
subject to the following obvious relation:  since $
W_{\kappa(q)}\subset \C^{q-1}$, the Whitney sum formula implies that
\[h_p(x_1,\dots, x_{\kappa(q)})=(-1)^pc_p(\C^{q-1}/W_{\kappa(q)})=0
\text{ if } p>q-\kappa(q)-1.\]
Thus the relations of $I_\kappa$
follow and are the only relations.
\end{proof}

Let us abuse notation, and use $e_\kappa$ to denote $
\Delta_*[Y_\kappa]$; this acts on $H_*(Y)$ by projection to $H_*(Y_\kappa)$.

\begin{lemma}\label{conv}
  The homology clases act on
  $H_*(Y,\K)$ by 
  \begin{align}
    \label{conv1}
    (h_1\times h_2)_* [Z_{i,\kappa,\kappa}]\star f(x_1,\dots,x_k)&=\frac{f(x_1,\dots,x_k)-f(x_1,\dots,x_{i+1},x_i,\dots,
      x_k)}{x_{i+1}-x_{i}}\\    \label{conv2}
(h_1\times h_2)_*[Z_{0,\kappa^+_i,\kappa}]\star
f(x_1,\dots,x_k)&=f(x_1,\dots,x_k)\\     \label{conv3}
 (h_1\times h_2)_*[Z_{0,\kappa^-_i,\kappa}]\star f(x_1,\dots,x_k)&=-x_if(x_1,\dots,x_k)
  \end{align}
\end{lemma}
\begin{proof}
  The correspondence $Z_{i,\kappa,\kappa}$ is a $\mathbb{P}^1$ bundle
  under both projections given by base change of the space of pairs of
  flags that have relative position $\leq s_i$.  Thus, the formula \eqref{conv1}
  follows from \cite[5.7]{BGGschub}.  If we have functions $\kappa\leq
  \kappa'$, then $\kappa'$ imposes a strictly weaker condition on
  flags; thus the correspondence $Z_{0,\kappa^+_i,\kappa} $ projects
  isomorphically to the first factor and  $Z_{0,\kappa^-_i,\kappa} $
  to the second.  Thus, the first correspondence induces a pullback
  and the second a
  pushforward in Borel-Moore homology.  The formula \eqref{conv2}
  follows from the fact that pullback sends fundamental classes to
  fundamental classes and commutes with cap product.  The formula
  \eqref{conv3} follows from the adjunction formula: the space
  $Z_{0,\kappa^-_i,\kappa} $ inside $X_{\kappa^-_i}$ is the zeroset of
  the induced map $W_i/W_{i-1}\to \C^{q}/\C^{q-1}$, so it is given by
  the Euler class of the line bundle $(W_i/W_{i-1})^*$, which is
  $-x_i$.  
\end{proof}

\begin{proof}[Proof of Theorem \ref{main}]
  First, note that we have a map $T^\ell_n\to A^\ell_n$ defined by the
  equations given in Theorem \ref{main}.  Both these can be identified
  with their image in the polynomial representations by Lemmata
  \ref{T-faithful} and \ref{A-faithful}.  The polynomial
  representations can be matched by Lemma \ref{GR}, and this
  intertwines the actions by Lemma \ref{conv}.  This map is thus also
  injective.  We only need to prove surjectivity.  We can do this by
  putting an upper bound on the dimension of $A^\ell_n$.  We can
  filter the variety $Y_{\kappa_1}\times_XY_{\kappa_2}$ according the
  preimages of the Schubert cells in $X$.  The Schubert cell has a
  free action by a unipotent subgroup of $GL_\ell$ (depending on the
  cell), and is thus an affine bundle over a single fiber.  Each
  Schubert cell contains a unique $T$-fixed point (here, $T$ is the torus of
  diagonal matrices), which is a coordinate subspace, spanned by the
  $(j+\la_j)$th coordinate vectors for $j=1,\dots, k$.  If we consider
  the fiber over this point, then it inherits an action of $T$, and
  the fixed points are given by pairs of flags of coordinate spaces on
  this space, with compatibility conditions with the standard flag
  specified by $\kappa_1$ and $\kappa_2$.  These are actually in
  bijection with pairs of backdrops whose associated functions are
  $\kappa_1,\kappa_2$.  The flag is given by adding coordinate vectors
  corresponding to the rows by reading them in the order specified by
  the backdrop.   

Thus, the $T$-fixed points of $Y\times_XY$ are in bijection with pairs
of backdrops on the same Young diagram; this gives an upper bound on
the sum of the Betti numbers, that is on the dimension of $A^\ell_n$.
However, this is the dimension of $T^\ell_n$ as computed by the basis,
so the map $T^\ell_n\to A^\ell_n$ must be surjective.
\end{proof}

\subsection{Relationship to sheaves}
\label{sec:relationship-sheaves}

While this is not necessary for understanding the overall
construction, the discussion of convolution algebras would be
incomplete without covering one of the prime motivations for
introducing them: their connection to the category of sheaves.  As
shown in \cite[8.6.7]{CG97}, the convolution algebra $A^\ell_n$ can also be
interpreted as an Ext algebra in the category of constructible sheaves
(or equivalently, D-modules) on the Grassmannian itself.  More
precisely, if $\K_Y$ is the sheaf of locally constant $\K$-valued functions:
\begin{proposition}
  $\displaystyle A^\ell_n\cong\operatorname{Ext}^\bullet(p_*\K_Y,p_*\K_Y)$.  
\end{proposition}
This Ext algebra completely controls the category of sheaves generated
by $p_*\K_Y$; there is a quasiequivalence of dg-categories between
the dg-modules over $A^\ell_n$ and the dg-category of sheaves
generated by $p_*\K_Y$.

Let us assume that $\K=\C$ (or more generally any field of
characteristic 0).
By the Decomposition Theorem of Beilinson, Bernstein, Deligne and
Gabber, the sheaf $p_*\C_Y$ is a sum of shifts
of simple perverse sheaves (see \cite{WhatisPS} for an introductory
discussion of this theory).  Replacing this sum with one copy of each
simple perverse constituent, we obtain an object $\mathcal{G}$ with
the property that $\mathsf{A}^\ell_n:=\Ext^\bullet(\mathcal{G},\mathcal{G})$ is a
positively graded algebra with its degree 0 part commutative and
semi-simple.  The algebras $A^\ell_n$ and $\mathsf{A}^\ell_n$ are
Morita equivalent since
they are Ext-algebras of objects with the same indecomposable
constituents.  

\begin{proposition}
The category of regular holonomic D-modules/perverse sheaves on the Grassmannian
$\Gr(k,\ell)$ which are smooth along the Schubert stratification is equivalent to the category of representations of the
Koszul dual of $\mathsf{A}^\ell_{\ell-2k}$ (the abelian category of
linear complexes of projectives $\mathsf{A}^\ell_{\ell-2k}$-modules).
\end{proposition}
\begin{proof}
Since the map from $Y\to \Gr(k,\ell)$ is equivariant for the subgroup
preserving the standard flag, every summand of  $\mathcal{G}$ is
smooth along the Schubert stratification. Since for every Schubert
cell, there's a $\kappa$ such that the Schubert cell is precisely the image of
$Y_\kappa$,  the IC sheaf of the Schubert cell is a summand of
$p_*\K_Y$ and thus of $\mathcal{G}$.

Thus, we have that Ext algebra of the sum of simple objects in this
category is $\mathsf{A}^\ell_{\ell-2k}$.
Since the category of perverse sheaves on the Grassmannian is Koszul,
it follows that $\mathsf{A}^\ell_{\ell-2k}$ is its Koszul dual.
\end{proof}
For a thorough discussion of Koszul duality, its relationship to
linear complexes, etc. see \cite{MOS}.  This result is particularly
interesting in view of the fact that this category already has an
algebraic description related to Khovanov homology, via work of
Khovanov on the arc algebra \cite{Kho02}. The category of
Schubert smooth perverse sheaves/D-modules on the Grassmannian is
equivalent to the parabolic category $\mathcal{O}$ for the
corresponding maximal parabolic by \cite[3.5.1]{BGS96} (interestingly,
this equivalence is {\it not}  simply taking sections of the D-module;
see \cite{WebWO} for a more detailed discussion).  
Of course,
those familiar with parabolic-singular duality for category $\cO$ (as
proven in \cite{BGS96}) will recognize that this implies that the
category of $A^\ell_n$-modules is equivalent to a certain block of
category $\cO$ of $\mathfrak{gl}_\ell$.  This is proven in \cite[\S
9]{Webmerged} by other methods (since the one used here is much harder to
generalize past $\mathfrak{sl}_2$), but this will perhaps not be too
meaningful to topologists.

However, this parabolic category $\cO$ (denoted $\mathcal{O}^{\ell-k,k}$ in \cite{StDuke}) played an
important role in the original definition of Khovanov's arc algebra  \cite{Kho02}.  The most
important case for understanding invariants is the $n=0$ weight space,
i.e. when $\ell=2k$; in this case, Stroppel
defined an extension $\mathcal{K}^k$ of Khovanov's arc algebra
\cite[\S 5.4]{Str09} which  has representation
category equivalent to $\mathcal{O}^{k,k}$ by \cite[5.8.1]{Str09}, and is thus Koszul dual to
$\mathsf{A}^\ell_{0}$.  That is:
\begin{theorem}\label{th:arc-Koszul}
  The algebra $A^{2k}_0$ is Morita equivalent to the Koszul dual of
  Stroppel's extended arc algebra $\mathcal{K}^k$.
\end{theorem}
A similar theorem holds for other weight spaces, using further
generalizations of the arc algebra given in \cite{BS1,ChK}.

\section{Khovanov homology}
\label{sec:khovanov-homology}

In order to define a knot homology, we need to define functors between
the categories of modules over $T^\ell$ for different choices of
$\ell$ corresponding to tangles.    These are defined explicitly using
bimodules over the algebras $T^\ell$.  

\subsection{Braiding}
\label{sec:braiding}
{\small \it 
\begin{quotation}
Spengler: [hesitates] We'll cross the streams.\\
Venkman: Excuse me, Egon? You said crossing the streams was bad!\\
\mbox{}[$\dots$]\\
Spengler: Not necessarily. There's definitely a {\em very slim}
chance we'll survive.\\

\hfill --Ghostbusters (1984)
\end{quotation}
}

The braiding bimodules are based on a simple principle used very
successfully in the movie ``Ghostbusters:'' even if you were told
not to do so earlier, you should ``cross the streams.''  
\begin{definition}
  A ${s_i}$-Stendhal diagram is collection of oriented curves which is
  a Stendhal diagram except that there is a single crossing between
  the $i$ and $i+1$st red strands.

  Let $\bra_i$ be the $T^\ell-T^\ell$-bimodule given by the quotient
  of the formal span of ${s_i}$-Stendhal diagrams by the same local
  relations (\ref{first-nilhecke}--\ref{cost}) as well as relations
  below (and their mirror images) \newseq
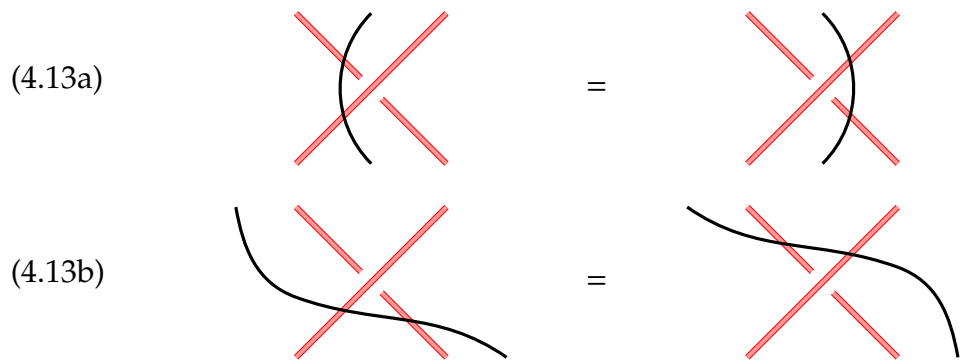
  \begin{equation*}\label{side-dumb}\subeqn
    \begin{tikzpicture}
      [very thick,scale=1,baseline] \usetikzlibrary{decorations.pathreplacing}
      \draw[wei] (1,-1) -- (-1,1) 
     node[midway,fill=white,circle]{};
      \draw[wei] (-1,-1) -- (1,1) ; \draw (0,-1)
      to[out=135,in=-135] (0,1); \node at (3,0){=}; \draw[wei] (7,-1)
      -- (5,1) node[midway,fill=white,circle]{}; \draw[wei] (5,-1)
      -- (7,1); \draw (6,-1) to[out=45,in=-45] (6,1);
    \end{tikzpicture}
  \end{equation*}
  \begin{equation*}\label{top-dumb}\subeqn
    \begin{tikzpicture}
      [very thick,scale=1,baseline] \usetikzlibrary{decorations.pathreplacing}
      \draw[wei] (1,-1) -- (-1,1) node[midway,fill=white,circle]{};
      \draw[wei] (-1,-1) -- (1,1); \draw (1.8,-1)
      to[out=145,in=-20] (-1,-.2) to[out=160,in=-80] (-1.8,1); \node
      at (3,0){=}; \draw[wei] (7,-1) -- (5,1) 
      node[midway,fill=white,circle]{}; \draw[wei] (5,-1) -- (7,1); \draw
      (7.8,-1) to[out=100,in=-20] (7,.2) to[out=160,in=-35] (4.2,1);
    \end{tikzpicture}
  \end{equation*}
This module is graded by giving each diagram a degree, which is the
sum of the number of red/black crossings plus twice the number of
dots, minus twice the number of black/black crossings\footnote{This is
slightly different from the grading convention in \cite{Webmerged};
since we'll be avoiding the discussion of ribbon structures, this
makes more sense for us.}. 
The bimodule action is by composition of diagrams, using the same
conventions as usual Stendhal
diagrams: the left action is by stacking diagrams from $T^\ell$ on top
of those from $\bra_i$, and the right action by placing them on the bottom.
\end{definition}
One example of such a diagram is   
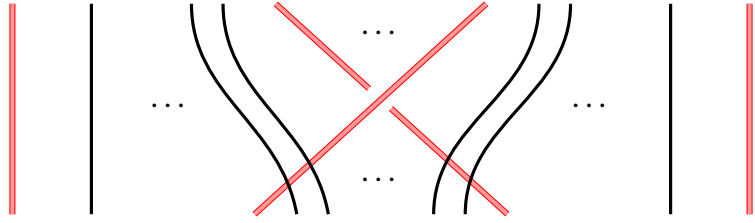
\begin{equation*}
    \begin{tikzpicture}
      [very thick,scale=1.4]
      \usetikzlibrary{decorations.pathreplacing} \draw[wei] (-3.5,-1)
      -- +(0,2); \draw (-2.75,-1) -- +(0,2); \node at
      (-2,0){$\cdots$}; \node at (2,0){$\cdots$}; \draw[wei] (3.5,-1)
      -- +(0,2); \draw (2.75,-1) -- +(0,2); \draw[wei]
      (1.2,-1) -- (-1,1) node[pos=.55,fill=white,circle]{}; \draw[wei]
      (-1.2,-1) -- (1,1); \draw (.5,-1) to[in=-90, out=90]
      (1.5,1); \draw (.8,-1) to[in=-90, out=90] (1.8,1); \draw
      (-.5,-1) to[in=-90, out=100] (-1.5,1); \draw (-.8,-1) to[in=-90,
      out=100] (-1.8,1); \node at (0,.7){$\cdots$}; \node at
      (0,-.7){$\cdots$};
    \end{tikzpicture}
  \end{equation*}

More generally, one can fix a permutation for the red lines to carry
out; the resulting bimodule $\bra_\sigma$ in this case will be the corresponding
tensor product of $\bra_i$'s for a reduced expression of the
permutation.  By \cite[6.5]{Webmerged}, the result is independent of
the choice of reduced expression.

The bimodule $\bra_\sigma$  has some beautiful properties:
\begin{itemize}
\item It has a cellular basis much like that of the algebra, indexed
  by pairs of backdrops on possibly different Young diagrams, defined
  in \cite[\S 3.4 \& 4.5]{WebRou}. We can
  define an action of $S_\ell$ on the Young diagrams in a $k\times
  (\ell-k)$ box via the rule $s_{\sigma\cdot \la}=\sigma\cdot s_\la$.
  That is, permutations act by reordering the line segments in the
  boundary of the permutation inside a box.  A simple permutation
  $s_i$ will add or remove a box if it switches a vertical and a
  horizontal segment, and leave the permutation unchanged if it
  switches two vertical segments or two horizontal ones.

The basis of $\bra_\sigma$ will be indexed by a pair of a backdrops
$\sT$ on
a partition $\la$, and $\sS$ on $\sigma\cdot \lambda$.  We define a $\sigma$-Stendhal diagram 
diagram $D_\sS$ where the top is the same as $B_\sS$, but the bottom
is given by $\kappa(\la)$ instead of $\kappa(\sigma\cdot \lambda)$.
The black strands at the bottom correspond to the parts of $\la$,
which are identified with the parts of $\sigma\cdot \la$ using the induced
permutation on vertical segments in the boundary, and thus to the
black strands at the top of the diagram.  As in $B_\sS$, we connect
black strands at the top and bottom which correspond to the same part
with a minimal number of crossings.

The desired bases are given by $D_\sS B_\sT^*$ (which gives a standard
filtration as a right module by \cite[4.14]{WebRou}) or its mirror
image (which gives a standard
filtration as a left module).
\item In particular, as both a left and as a right module, $\bra_\sigma$ has a
  filtration whose successive quotients are standard modules.
\item This bimodule has a geometric incarnation.  We constructed the
  varieties $Y$ using a chosen standard flag; let $Y'$ be the same
  variety, but defined using a different flag $U_\bullet$ such that
  $U_i$ is the span of the unit vectors $e_{\sigma(1)},\dots, e_{\sigma(i)}$.  In this
  case, we can canonically identify $H_*(Y'\times_XY')\cong T^\ell$,
  so $H^*(Y\times_XY')$ is a natural bimodule over $ T^\ell$; under
  the isomorphism of Theorem \ref{main}, this bimodule is isomorphic to $\bra_\sigma$ by \cite[4.12]{WebwKLR}. 
\end{itemize}

Given a bimodule $B$ over an algebra $A$ of finite global dimension, one can construct a functor
$A\gmod\to A\gmod$
from $B$ in two different ways.  You can consider the tensor product
$B\otimes_A-$, and the space $\Hom_{A}(B,-)$ of left module homomorphisms, which form an
adjoint pair; note that the {\it right} module structure on $B$ induces a {\it
  left} module structure on  $\Hom_{A}(B,-)$. The same is true of the derived functors on the
bounded derived category $D^b(A\gmod)$ of $A$-modules; we use
$\Lotimes$ to denote the left derived tensor product and $\RHom$ to
denote right derived homomorphisms:
\[B\Lotimes -\colon D^b(A\gmod) \to D^b(A\gmod)\qquad \RHom(B, -)\colon D^b(A\gmod) \to D^b(A\gmod).\]
If either one of these functors is an equivalence, the other one is
its inverse (up to isomorphism of functors).
Let $\mathbb{B}_i=\bra_i\Lotimes-$, and let $\sigma_i$ be the braid
making a positive crossing between the $i$ and $i+1$st strands of the
braid, as shown below:
\begin{equation}
  \tikz[baseline, very thick,scale=.5]{
\draw (-3.5,-1)
      -- +(0,2); \node at
      (-2,0){$\cdots$}; \node at (2,0){$\cdots$}; \draw (3.5,-1)
      -- +(0,2); \draw
      (1,-1) -- (-1,1) node[pos=.5,fill=white,circle]{}; \draw
      (-1,-1) -- (1,1); 
\node at (0,-2.2){$\sigma_i$};}
\qquad \qquad \tikz[baseline, very thick,scale=.5]{
\draw (-3.5,-1)
      -- +(0,2); \node at
      (-2,0){$\cdots$}; \node at (2,0){$\cdots$}; \draw (3.5,-1)
      -- +(0,2); \draw
      (-1,-1) -- node[pos=.5,fill=white,circle]{} (1,1); \draw
      (1,-1) -- (-1,1); 
\node at (0,-2.2){$\sigma_i^{-1}$};}\label{eq:sigma}
\end{equation}

\begin{theorem}[\mbox{\cite[6.18]{Webmerged}}]
  The assignment of the functors $\mathbb{B}_i$ to the braids
  $\sigma_i$ defines a strong action of the braid
  group on $\ell$ strands on the derived category $D^b(T^\ell\gmod)$.
\end{theorem}

This braid action is closely related to Khovanov
homology; there is a functor-valued invariant of tangles\footnote{Not
  quite the same as Khovanov's ``functor-valued invariant of
  tangles'' \cite{Kho02}.} which gives this action on braids and
Khovanov homology on links.  For lovers of category $\cO$, we can identify this with
natural representation-theoretic functors: if we identify $A^\ell_n\mmod$ with a block
of category $\cO$ which is ``submaximally singular'' then they match
with twisting functors and if we use the Koszul dual identification
with a regular block of parabolic category $\cO$, they match with
shuffling functors (this is proven in \cite[Th. C]{Webqui}).

\subsection{Cups and caps: $\ell=2$}

{\small \it 
\begin{quotation}
We are cups, constantly and quietly being filled. The trick is,
knowing how to tip ourselves over and let the beautiful stuff out.

\hfill --Ray Bradbury (1990)
\end{quotation}
}

In order to construct knot invariants, we need not just a braid group
action, but also a way of closing up our braids.  This is achieved by
defining functors corresponding to cups and caps.  Just as with the
braiding, these are fairly simple minded functors easily guessed by
drawing the appropriate pictures.  

As preparation, let's consider the case of a cup going from 0 strands
to 2.  In this case, we'll simply want a left module over $T^2_0$ which
categorifies the invariant vector in $\C^2\otimes \C^2$.  We'll use
the same notation here as in Section \ref{sec:an-example}.   Since the
functors $\fE$ and $\fF$ are exact, a module is killed by both of them
if and only if the same holds for all its composition factors.  

The algebra $T^2_0$ is 5 dimensional, and has 2 simple modules.  Since
the algebra is not semi-simple, this 
is only possible if both simples are one dimensional. Let $L_0$ be the simple quotient of $P_{(0,0)}$;
the idempotent $e_{(0,0)}$
acts by the identity on $L_0$.  Let $L_1$ be the simple quotient of
$P_{(0,1)}$; the idempotent $e_{(0,1)}$ acts by the identity on this
module.  To remind us of the relations we have imposed, we draw the
image of $e_{(0,1)}$ with a cup at the bottom as below:
\begin{equation*}
    \begin{tikzpicture}[very thick]
      \draw[wei] (0,1) to [out=-90,in=180] (.5,.2) to[out=0,in=-90]
      (1,1); \draw (.5,.2) to[out=90,in=-90] (.6,1); \node at
      (1.5,.6) {$=0$};
    \end{tikzpicture} 
  \end{equation*}
The relations imposed by killing the maximal submodule of $P_{(0,1)}$
are generated by:
\begin{equation*}
    \begin{tikzpicture}[very thick]
      \draw[wei] (0,1) to [out=-90,in=180] (.5,.2) to[out=0,in=-90]
      (1,1); \draw (.5,.2) to[out=90,in=-90] node
      [pos=.4,fill=black,circle,inner sep=1.5pt]{} (.6,1); \node at
      (1.5,.6) {$=0$};
    \end{tikzpicture} \qquad
    \begin{tikzpicture}[very thick]
      \draw[wei] (0,1) to [out=-90,in=180] (.5,.2) to[out=0,in=-90]
      (1,1); \draw (.5,.2) to[out=90,in=-90] (-.2,1); \node at
      (1.5,.6) {$=0$};
    \end{tikzpicture}
    \qquad
    \begin{tikzpicture}[very thick]
      \draw[wei] (0,1) to [out=-90,in=180] (.5,.2) to[out=0,in=-90]
      (1,1); \draw (.5,.2) to[out=90,in=-90] (1.2,1); \node at
      (1.5,.6) {$=0$};
    \end{tikzpicture}
  \end{equation*}

For any $T^2_0$-module, we have that $\fE M=e_{(0,0)}M$ by definition,
so $\fE L_0\cong \K$ and $\fE L_1\cong 0$.
An explicit calculation shows that $\fF L_0$ is a simple
module over $T^2_{-2}$ and $\fF L_1\cong 0$; this also follows because
we have an
isomorphism $\fF^{(2)}\fE M\cong \fE\fF^{(2)}M \oplus \fF M$ for any
object of weight 0 in a categorical $\mathfrak{sl}_2$ representation.  Thus, $L_1$ is the
desired ``invariant representation.'' 
\begin{proposition}
  The class $[L_1]$ in the Grothendieck group spans the
  space of $\mathfrak{sl}_2$-invariants
  $\bigwedge\nolimits^{\! 2}\C^2\subset \C^2\otimes \C^2$.
\end{proposition}
Thus, we'll want to think of the functor sending a $\K$-vector space $V$ to
$V\otimes_\K L_1$ as the cup functor going from 0 strands to 2.  

In this section, we will leave a number of statements for the reader
to verify; these results are all special cases of the results of
\cite{Webweb}.

There
are two obvious possibilities for the cap functor going from 2 strands
to 0, given by the right and left adjoints of the cup functor; the
right adjoint is $\RHom(L_1,-)$, and the left is $\dot{L}_1\Lotimes
-$, where $\dot{L}_1$ refers to the right module
obtained by letting $T^{\ell}$ act on $L_1$ via the mirror image
(through the $x$-axis) of diagrams (since $V\otimes_\K L_1\cong
\Hom_\K(\dot{L}_1,V)$).   These functors do not coincide, but they do
up to shift. 
Let $\langle n\rangle$ be the ``Tate twist'' which decreases the
internal grading of a module by $n$, and increases its homological
grading by $n$, that is $\langle n\rangle=\{n\}[-n]$.  
We will see below that  $\RHom(L_1,-)\langle -1\rangle \cong
\dot{L}_1\Lotimes -\langle 1\rangle $, and we will let this functor
correspond to the cap.   This is a special case of a more general
duality result \cite[3.17]{Webweb}.
 The isomorphism above, and many other special properties of
these cup and cap functors come from the special structure of a
projective resolution of $L_1$. Recall that in Section
\ref{sec:an-example}, we defined $\psi\in
e_{(0,1)}T^2e_{(0,0)}$ and $\phi\in e_{(0,0)}T^2e_{(0,1)}$ to be the unique
basis vectors in these spaces.   
\begin{proposition}
  The minimal projective resolutions of the simples $L_0,L_1$ are given by
  \[P_{(0,1)}\{-1 \}\overset{\psi}\to P_{(0,0)}\to L_0\qquad
  P_{(0,1)}\{ -2 \}\overset{\psi}\to P_{(0,0)} \{-1 \}\overset{\phi}\to
  P_{(0,1)}\to L_1\] where we use $\psi,\phi$ to indicate right
  multiplication by these elements.
\end{proposition}

In order to understand how the functors $\RHom(L_1,-)$ and
$\dot{L}_1\Lotimes-$ are related, we can try applying them to projectives.
Applying a right exact functor to a projective just gives a vector
space in homological degree $0$: thus, the projective $\dot{L}_1\Lotimes P_{(0,0)}$ is sent to 0, and
$\dot{L}_1\Lotimes P_{(0,1)}\cong \K$.  On the other hand, $\RHom(L_1,-)$ is left exact,
so we require the full projective resolution.  The result for any
module $P$ is the complex \[\RHom(L_1,P) = e_{(0,1)}P \overset{\phi}\to e_{(0,0)}P \{1\}\overset{\psi}\to e_{(0,1)}P\{ 2 \}\]
where the leftmost term is of homological degree 0 (so the rightmost
is of
homological degree 2) and now we are using the maps of left
multiplication by $\psi$ and $\phi$.  This
sends $P_{(0,0)}$ to 0 and $P_{(0,1)}$ to $\K\langle 2\rangle $.  We
want to emphasize that there is a symmetry being used here:  for
example \[\RHom(L_0,P) = e_{(0,0)}P \overset{\phi}\to e_{(0,1)}P\{ 1\} .\]
 
Phrased differently, we have shown that:
\begin{proposition}
  The cup and cap functors are biadjoint (up to shift).  
\end{proposition}

Every finite dimensional algebra $A$ has a Nakayama functor
$S(M):=A^*\Lotimes M$ where $A^*$ is the vector space dual of $A$
considered as a bimodule.  This functor sends the projective cover of any simple
object to its injective hull.  
The results above can be rephrased in terms of the Nakayama functor $S$
of $T^2_0$.  One can calculate that this functor sends
the projective resolution of $L_1$ to an injective resolution of $L_1$
(shifted so that the cohomology is in degree $-2$), whereas $L_0$ is
sent to a complex of injectives with cohomology in degrees $0$ and
$-1$.  

Since the algebra $T^{2}_0$ has finite global dimension (since it is
quasi-hereditary by Corollary \ref{cellular consequences}), its Nakayama functor is actually a right Serre
functor, i.e. we have a natural isomorphism for any $M,N$: \[\RHom(M,N)^*\cong \RHom(N,SM).\]  Thus, for any simple the relationship between $\RHom$ and
$\otimes$ is encoded by the fact that $\RHom(-,L)^*\cong
\dot{L}\Lotimes-$, and properties of a Serre  functor guarantee
\[\dot{L}\Lotimes-\cong \RHom(-,L)^*\cong \RHom(S^{-1}L,-).\]

Since $S^{-1}L_1\cong L_1\langle 2\rangle $, we obtain that 
\[\dot{L}_1\Lotimes-\cong \RHom(-,L_1)^*\cong \RHom(L_1,-)\langle -2\rangle .\]

One important consequence of this is that the coalgebra
$\dot{L}_1\Lotimes L_1$ and the algebra $\Ext^\bullet(L_1,L_1)$ are
identified with each other, giving a Frobenius structure on the
resulting space.  Of course, those familiar with Khovanov homology will
know what Frobenius structure to expect:
\begin{proposition}\label{S2}
  The Ext-bialgebra $\Ext^\bullet(L_1,L_1)$ is isomorphic to
  $H^*(S^2;\K)$ with its usual Poincar\'e Frobenius structure.
\end{proposition} 
This is also a special case of a more general result for
$\mathfrak{sl}_n$ \cite[3.20]{Webweb}.
This result holds over all fields, including those of characteristic
2.

\subsection{Cups and caps: $\ell>2$}
Now, let us turn to the more general case.  Now, we have $\ell$ red
strands, and expect to find functors either adding two more or capping
off two existing ones.  Furthermore, we expect it to be sufficient to
consider the cup functors, and that the caps will make their
appearance as adjoints.  

What we would like to find is a bimodule which ``inserts'' a copy of
$L_1$ with two new red strands attached to it. The beauty of using a
pictorial approach is that we can literally do exactly that; the
ugliness of a pictorial approach is that we then have to check a bunch
of relations to make sure we didn't just set everything to 0.

Let $+_i$ denote the tangle which (reading from the bottom) adds a cup between the $i$ and
$i+1$st strands, and $-_i$ its reflection in the vertical axis.
\begin{equation}
  \tikz[baseline, very thick]{
\draw (0,-.5) to[out=90,in=-90]  (-.5,.5); 
\node at (.25,0){$\cdots$}; 
\draw (1,-.5) to[out=90,in=-90]  (.5,.5);
\draw  (1.5,.5) to[out=-90, in=0] (1.25,.2) to [out=180,in=-90]  (1,.5);
\draw (1.5,-.5) to[out=90,in=-90]  (2,.5); 
\node at  (2.25,0){$\cdots$}; 
\draw (2.5,-.5) to[out=90,in=-90]  (3,.5); \node at (1.25,-1.1){$+_i$};}
\qquad \qquad \tikz[baseline, very thick]{
\draw (0,.5) to[out=-90,in=90]  (-.5,-.5); 
\node at (.25,0){$\cdots$}; 
\draw (1,.5) to[out=-90,in=90]  (.5,-.5);
\draw  (1.5,-.5) to[out=90, in=0] (1.25,-.2) to [out=180,in=90]  (1,-.5);
\draw (1.5,.5) to[out=-90,in=90]  (2,-.5); 
\node at  (2.25,0){$\cdots$}; 
\draw (2.5,.5) to[out=-90,in=90]  (3,-.5); \node at
(1.25,-1.1){$-_i$};}\label{eq:pm}
\end{equation}

More formally, let a $+_i$-Stendhal diagram be a diagram which follows
the Stendhal rules except that one of the red strands is a cup
connected to the top in the $i+1$st and $i+2$nd position at $y=1$;
this cup must have a unique minimum, and there is a black strand which
connects $y=1$ to this minimum. One example of a   $+_1$-Stendhal
diagram with $\ell=1$ is 
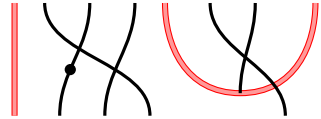
\begin{equation*}
  \begin{tikzpicture}[very thick,xscale=2,yscale=1.5]
\draw[wei] (-1,0) -- (-1,1);  \draw[wei] (0,1) to [out=-90,in=180] (.5,.2)
to[out=0,in=-90] (1,1);
\draw  (.5,.2) to[out=90,in=-90] (.6,1); 
\draw  (-.4,0) to[out=90,in=-90] (-.2,1); 
\draw  (.8,0) to[out=90,in=-90] (.3,1); 
\draw  (-.1,0) to[out=90,in=-90] (-.8,1);
\draw  (-.7,0) to[out=90,in=-90] node [pos=.4,fill=black,circle,inner sep=1.5pt]{}(-.5,1); 
  \end{tikzpicture} 
\end{equation*}
We can assign $+_i$-Stendhal diagrams a degree as usual, ignoring the
minimum; thus the diagram above with 3 black/black crossings, 1 dot,
and 1 red/black crossing has degree $-3$.

\begin{definition}
Let $\mathfrak{K}_i$ be the $T^{\ell+2}-T^\ell$-bimodule spanned over
$\K$ by  $+_i$-Stendhal diagrams modulo the local relations of
$T^\ell$ and the additional relations:
\newseq
  \begin{equation*}\subeqn\label{cup-1}
    \begin{tikzpicture}[very thick]
      \draw[wei] (0,1) to [out=-90,in=180] (.5,.2) to[out=0,in=-90]
      (1,1); \draw (.5,.2) to[out=90,in=-90] node
      [pos=.4,fill=black,circle,inner sep=1.5pt]{} (.6,1); \node at
      (1.5,.6) {$=0$};
    \end{tikzpicture} \qquad
    \begin{tikzpicture}[very thick]
      \draw[wei] (0,1) to [out=-90,in=180] (.5,.2) to[out=0,in=-90]
      (1,1); \draw (.5,.2) to[out=90,in=-90] (-.2,1); \node at
      (1.5,.6) {$=0$};
    \end{tikzpicture}
    \qquad
    \begin{tikzpicture}[very thick]
      \draw[wei] (0,1) to [out=-90,in=180] (.5,.2) to[out=0,in=-90]
      (1,1); \draw (.5,.2) to[out=90,in=-90] (1.2,1); \node at
      (1.5,.6) {$=0$};
    \end{tikzpicture}
  \end{equation*}
  \begin{equation*}\subeqn\label{cup-2}
    \begin{tikzpicture}[very thick]
      \draw[wei] (0,1) to [out=-90,in=180] (.5,.2) to[out=0,in=-90]
      (1,1); \draw (.5,.2) to[out=90,in=-90] (.6,1); \draw (-.1,0)
      to[out=90,in=-90] (1.2,1); \node at (1.5,.6) {$=$}; \draw[wei]
      (2,1) to [out=-90,in=180] (2.5,.2) to[out=0,in=-90] (3,1); \draw
      (2.5,.2) to[out=90,in=-90] (2.6,1); \draw (2.5,0)
      to[out=15,in=-90] (3.2,1);
    \end{tikzpicture} \qquad
    \begin{tikzpicture}[very thick]
      \draw[wei] (0,1) to [out=-90,in=180] (.5,.2) to[out=0,in=-90]
      (1,1); \draw (.5,.2) to[out=90,in=-90] (.6,1); \draw (1.1,0)
      to[out=90,in=-90] (-.2,1); \node at (1.4,.6) {$=-$}; \draw[wei]
      (2,1) to [out=-90,in=180] (2.5,.2) to[out=0,in=-90] (3,1); \draw
      (2.5,.2) to[out=90,in=-90] (2.6,1); \draw (2.5,0)
      to[out=165,in=-90] (1.8,1);
    \end{tikzpicture}
  \end{equation*}
The coevaluation functor $\mathbb{K}_i$ is given by
$\mathfrak{K}_i\Lotimes-$.  
\end{definition}
Of course, if $\ell=0$, then the resulting bimodule is just $L_1$.
The left and right adjoints of
$\mathbb{K}_i$ differ by the same shift as in the $\ell=0$ case (by \cite[3.17]{Webweb}).  Let \[\mathbb{E}_{i}:=\RHom(\mathfrak{K}_i,-)\langle -1\rangle \cong
\dot{\mathfrak{K}}_i\Lotimes -\langle 1\rangle.\]

As the case of $\ell=0$ shows, this is not an exact functor, but we
can do calculations with it by taking a projective resolution of
$\mathfrak{K}_i$ as a left module.  This can be done schematically as
follows:
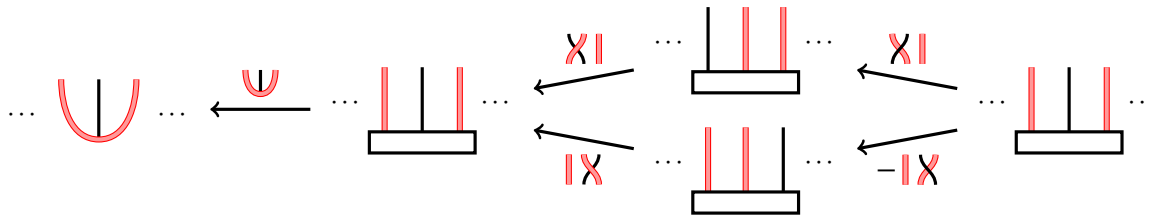
\begin{equation*}
 \begin{tikzpicture}[very thick,xscale=4.3,yscale=.8]
    \node (a) at (0,0){\begin{tikzpicture}[very thick] \node[scale=.8]  at (-.5,.5){$\cdots$}; 
      \draw (.5,.2) to[out=90,in=-90] (.5,1); \draw[wei] (0,1) to [out=-90,in=180] (.5,.2) to[out=0,in=-90]
      (1,1); \node[scale=.8]  at (1.5,.5){$\cdots$}; 
    \end{tikzpicture}};
    \node (b) at (1,0){\begin{tikzpicture}[very thick]\node[scale=.8]  at (-.5,.5){$\cdots$}; \node[scale=.8]  at (1.5,.5){$\cdots$}; 
      \draw[wei] (0,1) to (0,0); \draw[wei] 
      (1,1) to (1,0); \draw (.5,0) to[out=90,in=-90] (.5,1);
\node [draw=black,fill=white,inner xsep=20pt] at (.5,0){};
    \end{tikzpicture}};
    \node (c) at (2,1){\begin{tikzpicture}[very thick]\node[scale=.8]  at (-.5,.5){$\cdots$}; \node[scale=.8]  at (1.5,.5){$\cdots$}; 
      \draw (0,1) to (0,0); \draw[wei] 
      (1,1) to (1,0); \draw[wei] (.5,0) to[out=90,in=-90] (.5,1);\node [draw=black,fill=white,inner xsep=20pt] at (.5,0){};
    \end{tikzpicture}};
    \node (d) at (2,-1){\begin{tikzpicture}[very thick]\node[scale=.8]  at (-.5,.5){$\cdots$}; \node[scale=.8]  at (1.5,.5){$\cdots$}; 
      \draw[wei] (0,1) to (0,0); \draw
      (1,1) to (1,0); \draw[wei]  (.5,0) to[out=90,in=-90] (.5,1);\node [draw=black,fill=white,inner xsep=20pt] at (.5,0){};
    \end{tikzpicture}};
    \node (e) at (3,0){\begin{tikzpicture}[very thick]\node[scale=.8]  at (-.5,.5){$\cdots$}; \node[scale=.8]  at (1.5,.5){$\cdots$}; 
      \draw[wei] (0,1) to (0,0); \draw[wei] 
      (1,1) to (1,0); \draw (.5,0) to[out=90,in=-90] (.5,1);\node [draw=black,fill=white,inner xsep=20pt] at (.5,0){};
    \end{tikzpicture}};
\draw[->] (b) -- node[above, midway]{\begin{tikzpicture}[very thick,scale=.4,-]
      \draw (.5,.2) to[out=90,in=-90] (.5,1);\draw[wei] (0,1) to [out=-90,in=180] (.5,.2) to[out=0,in=-90]
      (1,1); 
    \end{tikzpicture}}(a);
\draw[->] (c) -- node[above, midway]{ \begin{tikzpicture}[very thick,scale=.4,-]
      \draw(0,1) to[out=-90,in=90] (.5,0); \draw[wei] 
      (1,1) to (1,0); \draw[wei]  (0,0) to[out=90,in=-90] (.5,1);
    \end{tikzpicture}}(b);  
\draw[->] (d) -- node[below, midway]{ \begin{tikzpicture}[very thick,scale=.4,-]
      \draw(1,1) to[out=-90,in=90] (.5,0); \draw[wei] 
      (0,1) to (0,0); \draw[wei]  (1,0) to[out=90,in=-90] (.5,1);
    \end{tikzpicture}} (b);
\draw[->] (e) -- node[above, midway]{ \begin{tikzpicture}[very thick,scale=.4,-]
      \draw[wei] (0,1) to[out=-90,in=90] (.5,0); \draw[wei] 
      (1,1) to (1,0); \draw  (0,0) to[out=90,in=-90] (.5,1);
    \end{tikzpicture} }(c);
\draw[->] (e) -- node[below, midway]{ $- \,\begin{tikzpicture}[very thick,scale=.4,-,baseline=2pt]
      \draw[wei] (1,1) to[out=-90,in=90] (.5,0); \draw[wei] 
      (0,1) to (0,0); \draw  (1,0) to[out=90,in=-90] (.5,1);
    \end{tikzpicture}$} (d);
\end{tikzpicture}
\end{equation*}
Here the boxes are there to fix the sequence at their top and impose
no other relations.  This is a complex of projective left modules;
there is no right action $T^\ell$ on each of the terms in this complex that commutes with the differentials. However,
by general nonsense there is an $A_\infty$-action of the algebra
$T^{\ell}$ on the complex, that is, an action
where the relations only hold up to an appropriate notion of homotopy
(see \cite[\S 2.3]{Webweb}).

What compatibility do we expect between these functors?  For any
composition of cups and caps, we have an associated functors, and we expect that any two ways of factoring a flat
$(p,q)$-tangle (that is, one with no crossings) as a composition of
functors will give isomorphic functors.  However, we expect much more
than this: the flat tangles form a 2-category, with morphisms given by
cobordisms.

In order to connect this construction to Khovanov homology, we use a
construction of Bar-Natan which defines a quotient of this category by
imposing additional relations. 
\begin{definition}
  We let $\mathcal{{BN}}$ be the 2-category dotted cobordism
  category with the relations given in \cite[\S
11.2]{BarN05}.  The objects of this category are
  non-negative integers, its 1-morphisms are flat tangles and its
  2-morphisms are cobordisms decorated with dots modulo Bar-Natan's ``sphere,'' ``torus''
  and ``neck cutting'' relations.  
\end{definition}
Note that this 2-category is not quite what Bar-Natan considers in
\cite{BarN05}; he considers a ``canopolis'' which contains a more
flexible and general notion of composition. For our purposes, it seems
to be necessary to use this more restrictive framework.

There is also a graded version of this 2-category where the
1-morphisms are formal grading shifts of flat tangles, and the
morphisms are cobordisms of degree 0 (with grading shifts accounted for). Following the 
\cite[\S 6]{BarN05}, the notation $\{m\}$ means decreasing the grading by
$m$; that is, a morphism $T\to T' \{m\}$ in the graded category is one
of degree $m$ in the ungraded category.

\begin{theorem}[\mbox{Chatav \cite[\S 4.1]{Chatav}; Mackaay-W \cite[4.21]{Webweb}}]\label{th:Chatav}
 The functors $ \mathbb{K}_i$ and $\mathbb{E}_{i}$ define a strict 2-representation $\gamma$ of the Bar-Natan 2-category
 $\mathcal{{BN}}$ such that
 \begin{itemize}
 \item Each integer $\ell$ is sent to the category of modules over
   $T^\ell$: we have that $\gamma(\ell)=D^b(T^\ell\gmod)$.
\item The cup tangles and cap tangles 
   are sent to $\mathbb{K}_i$ and $\mathbb{E}_{i}$: we have that
   $\gamma(+_i)=\mathbb{K}_i$ and $\gamma(-_i)=\mathbb{E}_i$. 
\item Cobordisms corresponding to handle attachments are sent to the
  unit or counit of the appropriate adjunction between $\mathbb{K}_i$
  and $\mathbb{E}_{i}$.
 \end{itemize}
This 2-representation can be extended to the graded Bar-Natan category
intertwining the Tate twist $\langle m\rangle$ in $D^b(T^\ell\gmod)$ with
the grading shift $\{m\}$ in the graded Bar-Natan 2-category.
\end{theorem}

Note that in this context, Bar-Natan's relations actually follow immediately from
Proposition \ref{S2}, since these relations just express the structure
of the cohomology ring $H^*(S^2;\K)$.  Bar-Natan's
relations then just specify that if $t$ is the unique element of
degree 2 with trace 1, then this element has square 0, and that
the dual ordered basis to $\{t,1\}$ under the Frobenius trace is $\{1,t\}$.

\subsection{Comparison with Khovanov homology}
\label{sec:comp-with-khov}

The calculations we have done up to this point suggest an approach to
finding a knot invariant, or more generally a tangle invariant. As
is often necessary in quantum topology, we will choose a generic
tangle projection and perform a construction using it which ultimately we
can see is independent of the choice.  If we slice this tangle
projection along horizontal lines, we can cut it up into simple pieces
each with one of the following forms, shown in the equations \eqref{eq:sigma}
and \eqref{eq:pm}:
\begin{itemize}
\item a positive crossing $\sigma_i$ of the $i$th and $i+1$st strands,
\item a negative crossing $\sigma_i^{-1}$ of the $i$th and
  $i+1$st strands,
\item  a cup  $+_i$ appearing between the $i$th and $i+1$st
  strands, or 
\item a cap $-_i$ joining the $i+1$th and $i+2$st strands.
\end{itemize}
We will define a functor $\mathscr{K}$ such that:
\newseq
\begin{align*}
\mathscr{K}(\sigma_i)&=\mathbb{B}_i\langle 1\rangle&\mathscr{K}(\sigma_i^{-1})&=\mathbb{B}_i^{-1}\langle -1\rangle\subeqn\label{K-def1}\\
\mathscr{K}(+_i)&=\mathbb{K}_i&\mathscr{K}(-_i)&=\mathbb{E}_{i}\subeqn\label{K-def2}.
\end{align*}
For any $(p,q)$-tangle $\mathscr{T}$, we
choose a generic projection, cut into these pieces and define
$\mathscr{K}(\mathscr{T}) \colon T^p\gmod \to T^q\gmod$ by composing
the functors associated to the pieces by (\ref{K-def1}--\ref{K-def2}).
Note, we are using unoriented knots; ``positive'' and ``negative'' as
used above are relative to the $y$-coordinate in the plane (either
both strands upward or downward oriented).  For the moment, ignore
that this depended on a choice of projection.  

While what we have written thus far points naturally to this
definition, it's not completely satisfactory. It doesn't have an obvious
  connection to Khovanov homology, nor have we checked that it defines
  a tangle invariant (that it doesn't depend on the choice of
  projection). 

However, we have an alternate definition of a knot invariant which fixes
both these problems: we could simply transport structure from
Bar-Natan's paper. That is, if we have a tangle with no crossings,
then the corresponding functor is that of Theorem \ref{th:Chatav}, and
for $\sigma_i$, we take the image under the
2-functor $\gamma$ of a particular complex in Bar-Natan's cobordism
category, given by the saddle cobordism from the identity to the
composition of a cap and cup.  

If we consider
the complex $[\mathcal{T}]$ associated to a tangle $\mathcal{T}$ in
Bar-Natan's construction \cite[\S 2.8]{BarN05} or more precisely its
graded version defined in \cite[6.4]{BarN05}, its image $\gamma( [\mathcal{T}])$ is a
complex of functors $D^b(A^p\gmod)\to D^b(A^q\gmod)$; we can take
iterated cone of this complex\footnote{Technically, one should keep track of a
  dg-enhancement in order to make sense of this iterated cone, but we
  just use the standard one on any derived category of an abelian
  category with enough projectives.  That is, we always just replace
  everything with its projective resolution; any morphism in the
  derived category lifts to a chain map between projective resolutions,
and we can take the cones of these.} to get a single functor $D^b(A^p\gmod)\to D^b(A^q\gmod)$, which
we'll also denote $\gamma( [\mathcal{T}])$ by an abuse of notation.
This is a tangle invariant, since the homotopy type of $[\mathcal{T}]$
is a tangle invariant by \cite[Th. 1]{BarN05}.

Since both $\gamma(
[\mathcal{T}])$ and $\mathscr{K}(\mathcal{T})$ are functorial under
 tangle composition, its enough to check that they coincide on the cup,
cap and crossing tangles considered earlier.  This follows by
definition for the cup and cap.  Thus it only needs to be checked for
the crossing:
\begin{theorem}\label{comparison}
The 2-functor $\gamma$ sends the cone of the crossing complex in
$\mathcal{{BN}}$ to the functor $\mathbb{B}_i$. More generally, $\gamma(
[\mathcal{T}])\cong \mathscr{K}(\mathcal{T})$.
  \end{theorem}

Consider the action of Bar-Natan's positive crossing: this is the cone
of a map between two functors, the identity functor and
$\mathbb{E}\mathbb{K}\langle 1\rangle $.  In fact, both of these
correspond to derived tensor product with honest bimodules, given by
the algebra $T^\ell$ itself, and the second by
$\mathfrak{K}_i\otimes_{T^{\ell}} \dot {\mathfrak{K}}_i$. Thus, the
image of the crossing under $\gamma$ is the cone of the unit $\upsilon$ of
the adjunction $(\mathbb{E},\mathbb{K}\langle 1\rangle )$.

The counit of this adjunction is given by the pairing $\mathfrak{\dot
  K}_i\otimes_{T^\ell} \mathfrak{K}_i\to T^{\ell-2}$ where we stack the
diagrams, and simplify using the relations of $T^\ell$.  The result is
a Stendhal diagram with a single red circle which we evaluate by sending the
``theta'' diagram to the empty diagram:
\[\tikz[very thick, baseline=-2pt]{\draw (0,-.5)--(0,.5); \draw [wei] (-.5,0) to[out=90,in=180] (0,.5);
  \draw [wei] (-.5,0) to[out=-90,in=180] (0,-.5); \draw [wei] (.5,0)
  to[out=90,in=0] (0,.5); \draw [wei] (.5,0) to[out=-90,in=0]
  (0,-.5); }\mapsto \emptyset.\]
This rule together with the relations (\ref{cup-1}--\ref{cup-2}) allow
us to simplify to a diagram in $T^{\ell-2}$.  

Thus the  unit is given by
sending the identity $1\in T^\ell$ to the canonical element of this pairing.  This is given by the sum of all diagrams with no
crossings, and a single pair of red cups and caps with one black strand
inside the cup and inside the cap.  We can evaluate any other element of the
algebra by multiplying the image of the identity on the left or
right.  Note that any idempotent which does not have exactly 1 black strand
between these two reds will kill this element and thus be sent to
zero.  For example 
\[\tikz[baseline=-2pt,very thick]{\draw[wei] (-.5,-.5) -- (-.5,.5);  \draw[wei]
  (0,-.5) -- (0,.5);\draw (.5,-.5) -- (.5,.5);}\,\mapsto \, 0 \qquad
\qquad \tikz[baseline=-2pt,very thick]{\draw[wei] (-.5,-.5) -- (-.5,.5);  \draw
  (0,-.5) -- (0,.5);\draw[wei] (.5,-.5) -- (.5,.5);}  \,\mapsto \,
\tikz[baseline=-2pt,very thick]{\draw
  (0,-.5) -- (0,-.2);\draw (0,.2) -- (0,.5);\draw[wei] (-.5,-.5) to[out=90,in=180] (0,-.2) to[out=0,in=90] (.5,-.5);  \draw[wei] (-.5,.5) to[out=-90,in=180] (0,.2) to[out=0,in=-90] (.5,.5); } \]
In general, this evaluation can proceed by fixing some horizontal
slice $y=a$ and pinching the $i+1$st and $i+2$nd  red strands together to make a cup
and cap;  if for any $a\in [0,1]$ there is not exactly 1 black strand
between these two reds  at $y=a$, we get 0.  

On the other hand, we have a natural map $\psi\colon \bra_i\to T^\ell$ given by
using the ``0-smoothing'' of the red crossing, that is slicing
vertically through the red crossing in order to produce two strands
with no crossing.  
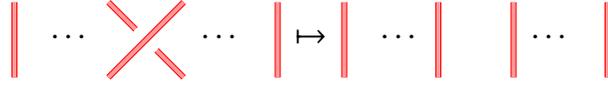
\[  \tikz[baseline=-2pt, very thick,scale=.5]{
\draw[wei] (-3.5,-1)
      -- +(0,2); \node at
      (-2,0){$\cdots$}; \node at (2,0){$\cdots$}; \draw[wei] (3.5,-1)
      -- +(0,2); \draw[wei]
      (1,-1) -- (-1,1) node[pos=.5,fill=white,circle]{}; \draw[wei]
      (-1,-1) -- (1,1); }\,\mapsto \,\tikz[baseline=-2pt, very thick,scale=.5,wei]{
\draw[wei] (-3.5,-1)
      -- +(0,2); \node at
      (-2,0){$\cdots$}; \node at (2,0){$\cdots$}; \draw[wei] (3.5,-1)
      -- +(0,2); \draw[wei]
      (1,-1) -- (1,1); \draw[wei]
      (-1,-1) -- (-1,1); 
}\]
This is obviously compatible with the relations and
injective. It's image is killed by $\upsilon$, since doing the
``pinch'' at the $y$-value where the 0-smoothing occurs gives two red
strands not separated by a black, and thus 0.    Thus, we will complete
the proof of Theorem \ref{comparison} by showing:

\begin{lemma}
  The map $\psi$ induces an isomorphism $\bra_i\cong \ker \upsilon$.  
\end{lemma}
\begin{proof}
 We can reduce to the case where $\ell=2$ using
 \cite[6.9 \& 7.19]{Webmerged}. Assuming $\ell=2$, this is a simple calculation; one
  simply notes that both $\bra_i$ and $\ker \upsilon$ are 4
  dimensional.   There is a basis of $\bra_i$ (compatible with the
  cellular filtration as a left module) which is
  given by
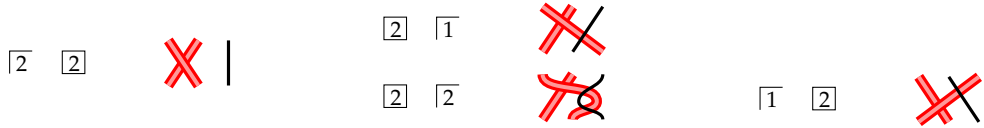
\[ \tikz{ 
\node at (-5,0){\tikz[xscale=.8,yscale=.6]{
\node at (-2,.5){\tikz{\draw (.3,.3) -- (0,.3) -- (0,0);\draw(.7,.3) --
    (.7,0) -- (1,0) --(1,.3)--cycle; \node[scale=.7] at (.16,.14) {$2$};\node[scale=.7] at (.86,.14) {$2$};}};
\draw[wei,very thick] (0,0)--(.5,1);
    \draw[wei,very thick] (.5,0) --(0,1);\draw[very thick] (1,0)--(1,1); }};
\node at (0,0){ \tikz[xscale=.8,yscale=.6]{
\node at (-2,-1){\tikz{\draw(0,.3) -- (0,0) -- (.3,0) --(.3,.3)--cycle;\draw  (1,.3) --(.7,.3) --
    (.7,0);\node[scale=.7] at (.16,.14)
    {$2$};\node[scale=.7] at (.86,.14) {$1$};}};
\node at (-2,-2.5){\tikz{\draw(0,.3) -- (0,0) -- (.3,0) --(.3,.3)--cycle;\draw  (1,.3) --(.7,.3) --
    (.7,0); \node[scale=.7] at (.16,.14) {$2$};\node[scale=.7] at (.86,.14) {$2$};}};
\draw[wei,very thick] (0,-1.5)--(.5,-.5);
    \draw[wei,very thick] (1,-1.5) --(0,-.5);\draw[very thick] (.5,-1.5)--(1,-.5);
\draw[wei,very thick] (0,-3)--(.5,-2);
    \draw[wei,very thick] (.5,-3) to[out=90,in=-90] (.9,-2.6) to[out=90,in=-90] (0,-2);\draw[very thick] (1,-3)
    to [out=90,in=-90] 
(.6,-2.6) to [out=90,in=-90] 
 (1,-2); }};
\node at (5,0){ \tikz[xscale=.8,yscale=.6]{\node at (0,1){\mbox{}};
\draw[wei,very thick] (.5,-1.5)--(0,-.5);
    \draw[wei,very thick]
    (0,-1.5)--(1,-.5);\draw[very thick] (1,-1.5) --(.5,-.5);
\node at (-2,-1){\tikz{\draw (.3,.3) -- (0,.3) -- (0,0);\draw(.7,.3) --
    (.7,0) -- (1,0)--(1,.3)--cycle; \node[scale=.7] at (.16,.14) {$1$};\node[scale=.7] at (.86,.14) {$2$};}};
}};} \]
These are sent under the map breaking open the crossing to 4 of the 5
basis vectors shown in Section \ref{sec:an-example}, which necessarily
span the kernel of $\upsilon$.
\end{proof}

\begin{proof}[Proof of Theorem \ref{Khovanov}]
  By \cite[Th. 1]{BarN05}, for a link $\mathcal{L}$, we have that
  $[\mathcal{L}]$ is just the Khovanov homology
  $\operatorname{Kh}(\mathcal L)$ of this link, tensored with the
  empty diagram (the derivation of this result using delooping is actually explained more clearly in
  \cite{BNfast}).  Thus
  $\gamma( [\mathcal{L}])$ is an endofunctor of  $D^b(T^0\gmod)\cong D^b(\K\gmod)$ 
given by tensor product with $\operatorname{Kh}(\mathcal L)$,
  thought of as a complex of graded vector spaces, though with
  slightly different grading, since the internal grading in
  Bar-Natan's picture is sent to the Tate twist in our grading.  Thus,
  the same is true for $\mathscr{K}(\mathcal L)$ by Theorem \ref{comparison}.   
\end{proof}

The readers familiar with the literature on Khovanov homology might
get a bit nervous around this point: though Bar-Natan's construction is
beautiful, it has a well-known flaw: it only allows one to define
functoriality maps on Khovanov homology up to sign.  However, a fix
for this issue was found by Clark, Morrison and Walker \cite{CMW} and
can be transported into our picture in a straightforward way.  Recall that our
identification with Khovanov homology involved considering a map
$\bra_i\to T^\ell$ and identifying its cokernel with
$\mathfrak{K}_i\otimes_{T^{\ell}} \dot {\mathfrak{K}}_i$.  While these
modules are isomorphic, they are not {\it canonically} isomorphic.
Rather than taking the obvious identification, one should insert
factors of $i$ or $-i$ to account for orientations.  We leave to the
reader the details of transporting the disoriented Bar-Natan category
using this approach.

\subsection{Jones-Wenzl projectors}
\label{sec:signs}

Another construction in the category $\mathcal{{BN}}$ which we
would like to understand in terms of $T^\ell$ is the categorified Jones-Wenzl
projector $\mathsf{P}_\ell$ of
Cooper and Krushkal \cite{CoKr}.  Much like the crossing, we can
easily transport this structure to an endofunctor using the 2-functor
$\gamma$; however, since this complex is unbounded, it induces a
autofunctor on the bounded above derived category $D^-(T^\ell\gmod)$ of graded
$T^\ell$ modules\footnote{Actually, there are dual categorical
  Jones-Wenzl projectors, one bounded above and one bounded below as
  complexes.  We'll always use the bounded above one.}.

Each algebra
  $T^\ell_{\ell-2k}$ has a single indecomposable
  projective-injective; this is given by a divided power functor
$\fF^{(k)}P_\emptyset$.   Under the correspondence of indecomposable
projectives to parity sheaves on the Grassmannian
$\Gr(k,\ell)$, the object $\fF^{(k)}P_\emptyset$ is sent to the
constant sheaf $\K_{\Gr(k,\ell)}$. Thus the endomorphism ring 
$\End(\fF^{(k)}P_\emptyset)$ is isomorphic to the cohomology
ring of $H^*(\Gr(k,\ell);\K)$.  We can also establish this
algebraically, since symmetric polynomials in the dots span
$\End(\fF^{(k)}P_\emptyset)$ and we must have that $h_m(\mathbf{y})=0$ for
$m>\ell-k$ since $I_0$ acts trivially.  This defines a surjective map
$H^*(\Gr(k,\ell);\K)\to \End(\fF^{(k)}P_\emptyset)$ which a dimension
count shows must be an isomorphism\footnote{One can also derive this
  using the equivalence to a singular block of category $\cO$ and
  Soergel's Endomorphismensatz \cite{Soe90}.}.
\begin{definition}
  We let $\mathcal{S}_0$ be the subcategory of $D^-(T^\ell\gmod)$
  consisting of complexes of projective-injectives. 
\end{definition}

This subcategory has an
orthogonal $\mathcal{S}_0^\perp$, given by the objects whose
composition factors are all killed by $\fE^k$.  Typically, one has to
specify left or right orthogonals in a categorical setting, but in
this case, these coincide.

Since the left and
right orthogonals coincide, there is a unique projection $\pi_\ell$ to
$\mathcal{S}_0$ killing this orthogonal.  A similar projection on
blocks of category $\cO$ is considered in \cite[\S 8]{FSS}; their projection
is intertwined with $\pi_\ell$ by the equivalence of $A^\ell_n\mmod$ to
a block of category $\cO$ discussed in Section
\ref{sec:relationship-sheaves}.

This may sound like an abstract operation, but in terms of algebras,
it's really very concrete. Recall the idempotent $e_0$ defined in
Section \ref{sec:stendhal-diagrams}.  Consider the
bimodule $T^\ell e_0 T^\ell\subset T^\ell$ .  Essentially by
definition, this is the bimodule of diagrams as in $T^\ell$
which have all black strands right of all red at $y=\nicefrac 12$.
\begin{lemma}
  The projection functor $\pi_\ell$ coincides with the derived tensor product $T^\ell e_0 T^\ell\Lotimes-$
\end{lemma}
\begin{proof}
This follows immediately from the fact that
 the category $\mathcal{S}_0$ is generated
by the summands of $T^\ell e_0$.
\end{proof}

One can think of this as the composition of two adjoint
functors. Recall that $R^\ell=e_0T^\ell e_0$ is isomorphic to the
cyclotomic nilHecke algebra with a degree $\ell$ relation, via the map
that puts a nilHecke diagram to the right of $\ell$ red lines.  
We thus have a functor $M\mapsto e_0M$ which sends $T^\ell\gmod$ to
$R^\ell\gmod$, and its left adjoint $T^\ell e_0\Lotimes_{R^\ell}-$;
taking derived tensor product is necessary since $T^\ell e_0$ is not
projective as a right $R^\ell$-module.

\begin{lemma}\label{perp-hw}
  The category $\mathcal{S}_0^\perp$ is the smallest triangulated
  subcategory of $D^-(T^\ell\gmod)$ which is closed under
  categorification functors $\fE,\fF$ and
  contains all highest weight simples of weight $<\ell$.  
\end{lemma}
\begin{proof}
 If $M$ is a module of weight $m$, we can identify
 $e_0M=\fE^{(\ell-m)/2}M$; thus all highest weight simples of weight
 $<\ell$ lie in $\mathcal{S}_0^\perp$.  On the other hand,
 $\mathcal{S}_0^\perp$ is equivalent to the quotient
 $D^-(T^\ell\gmod)/\mathcal{S}_0$, which is concentrated in weights strictly
 between $\ell$ and $-\ell$.  Thus, it is generated by its highest
 weight simples, which all necessarily of weight $<\ell$.  Thus, the
 same is true of $\mathcal{S}_0^\perp$.
\end{proof}

Obviously, $T^{m}_m$ has a unique highest weight simple, which we
denote $P_\emptyset$.

\begin{lemma}\label{cup-simple}
    The images of $P_\emptyset$ under the different 
flat $(\ell,m)$ tangles with no caps are a complete, irredundant list of highest
weight simples of weight $m$.
\end{lemma}
\begin{proof}
  Since the cup functors intertwine the categorification functors
  $\fE$ and $\fF$, the
  image of $P_\emptyset$ under any flat tangle is highest weight.  In
  particular, any composition factor of such a module is highest
  weight.

We attach a 
sign sequence to a flat $(\ell,m)$ tangle $\mathcal{T}$ with no caps above by
putting a $+$ above each stand which goes from the bottom to the top
and over the right end of each cup, and a $-$ over the left end of
each cup.  We can consider this sequence as an element of the tensor product
of $\ell$ copies of the two-element crystal $\{+,-\}$ of
$\mathbb{C}^2$ (see \cite[\S 4.4]{HongKang}). In this crystal, the sequence is highest weight,  as there is no $-$ sign not
canceled by a $+$ to its right.  The action of the Kashiwara operator $\tilde{e}_i$
on the weight string generated by this element 
changes the rightmost $-$ on top of a through-strand to a $+$, leaving
the cups unchanged.

We can associate an idempotent $e_{\mathcal{T}}$ in $T^\ell$ to a flat $(\ell,m)$ tangle with no caps ${\mathcal{T}}$: we
replace each $-$ by a red strand with a black to its right, and each
$+$ by just a red strand.  We equip the set of these sign sequences
with a partial order by the rule
that $-+>+-$.    We can convert a sign sequence to a partition  in a
box $\la_{\mathcal{T}}$, by
replacing each $+$ with a horizontal line segment and each $-$ by a
vertical line segment, and considering this as the boundary of the
Young diagram (as in Section \ref{sec:decategorification}).
In this case, we have that  $\la \geq \mu$ if the diagram of $\la$
fits inside that of $\mu$, which is the same as the order on cells in
the cellular basis of Section \ref{sec:natural-basis}.
By Corollary  \ref{cellular consequences}, there is a unique highest weight simple such that
$\dim e_{\mathcal{T}}L_{\mathcal{T}}=1$ and $e_{{\mathcal
    T}'}L_{\mathcal T}=0$ for ${\mathcal T}'>{\mathcal T}$, which is
the unique simple quotient of the corresponding cell module $S_{\la_{\mathcal T}}$.

Consider the module $K_{\mathcal T}:=\mathscr{K}({\mathcal T})(P_\emptyset)$.  We can easily calculate that $\dim e_{\mathcal T}K_{\mathcal T}
\leq 1$,
since this space is spanned by the diagrams 
where the black strand from each cup follows the left side up to the
top.  In one example, this is the resulting diagram:
\begin{equation}\label{cups}
     \begin{tikzpicture}[very thick,baseline]
      \draw (.5,.4) to[out=90,in=-90] (.3,1); 
\draw (2.5,.4) to[out=90,in=-90] (2.3,1); 
\draw (1.25,-.4) to[out=165,in=-90] (-.85,1); 
\draw[wei] (0,1) to [out=-90,in=180] (.5,.4) to[out=0,in=-90]
      (1,1); \draw (.5,0) to[out=130,in=-90] (-.3,1); 
    \draw[wei] (-.5,1) to [out=-90,in=180] (.5,0) to[out=0,in=-90]
      (1.5,1);
\draw[wei] (2,1) to [out=-90,in=180] (2.5,.4) to[out=0,in=-90]
      (3,1);
         \draw[wei] (-1,1) to [out=-90,in=180] (1.25,-.4) to[out=0,in=-90]
      (3.5,1);
\end{tikzpicture}
\end{equation}
Any other diagram $d$ in  $e_{\mathcal T}K_{\mathcal T}$ must have a black strand which
passes through the left side of its cup.  Using the relations, we can
push this crossing lower, until it is the first crossing on this black
strand.  Correction terms will appear from (\ref{red-triple-correction}), but
these will have fewer red/black crossings.  The relations
(\ref{cup-1}) imply that the diagram where  the black strand passes
through the left side of the cup is 0, so we can write $d$ as a
sum of diagrams with fewer red/black crossings.  By induction, we may
assume that there are no such crossings, and indeed the diagram we indicated in
\eqref{cups} spans.

Furthermore, this diagram generates the module $K_{\mathcal T}$; in order to see
this, pull the bottom of each cup toward the bottom of the diagram,
making sure its minimum ends up to the right of the black strand for
any cup in which it is nested.  Eventually you will reach a Stendhal
diagram applied to $e_{\mathcal T}K_{\mathcal T}$.  Since the module $K_{\mathcal T}$ is not zero (it
categorifies a non-zero vector), this shows that $\dim e_{\mathcal T}K_{\mathcal T}=1$.

An argument like that above shows that $e_{{\mathcal T}'}$ with ${\mathcal T}'>{\mathcal T}$ kills this module,
since there is no diagram with the correct top which doesn't have a
black strand passed through the left side of its cup.  Thus, $L_{\mathcal T}$
must be a quotient of $K_{\mathcal T}$.

The module $K_{\mathcal T}$ is self-dual, so $L_{\mathcal T}$ also appears as a submodule.  Since
$\dim e_{\mathcal T}K_{\mathcal T}=1$, this is only possible if
$K_{\mathcal T}=L_{\mathcal T}$. We see from Corollary \ref{cellular
  consequences} that if $L_{\mathcal T}\cong L_{\mathcal T'}$, then we
must have that $\la_{\mathcal T}= \la_{\mathcal T'}$.    Since different sign sequences
result in different partitions, we must also have that
$\mathcal T=\mathcal T'$, which proves the desired irredundancy.
\end{proof}

\begin{theorem}\label{one-projector}
  The categorified Jones-Wenzl projector $\mathsf{P}_\ell$ is sent by $\gamma$ to the projection
  $\pi_\ell=T^\ell e_0 T^\ell\Lotimes_{T^\ell}-$ to
  the subcategory $\mathcal{S}_0$.
\end{theorem}

\begin{proof}
The projection is distinguished by the fact that it is isomorphic to
the identity functor on $\mathcal{S}_0$ and kills all objects in
$\mathcal{S}_0^\perp$.  Thus, we need only check that $\mathsf{P}_\ell$ also has these properties.  

The images of all 1-morphisms in $\mathcal{{BN}}$ commute with the
functors $\fE$ and $\fF$ by Theorem \ref{th:Chatav}. Since $\mathcal{S}_0$ is generated by
$\fF^kP_\emptyset$ and $\mathsf{P}_\ell$ acts by the identity
on $T^\ell_\ell\gmod$,  it also acts by the identity on $\mathcal{S}_0$.

On the other hand, $\mathsf{P}_\ell$ kills the image of
any cup functor, since it is invertible under turn-backs.  Thus, by
Lemma \ref{cup-simple}, it kills all highest weight simples of weight
$<\ell$. Since it commutes with categorification functors, it kills
the triangulated category generated by categorification functors
applied to these simples. In turn,
by Lemma \ref{perp-hw}, this category is $\mathcal{S}_0^\perp$.  This
completes the proof.
\end{proof}
 In \cite[\S 8]{Webmerged}, we define a homology theory categorifying the
 colored Jones polynomial which uses generalizations of the algebras
 $T^\ell$.  For each sequence of positive integers
 $\mathbf{n}=(n_1,\dots, n_m)$ with $\ell=\sum n_i$, we have an idempotent $e_{\mathbf{n}}$
 which is the sum of all idempotents where there is a group of $n_1$ red
 strands, then some number of black strands, $n_2$ red strands, etc.
 In terms of $\kappa$, this means that the first $n_1$ values of
 $\kappa$ are the same, then the next $n_2$, etc.  
That is, we have the sum of the idempotents associated to any sequence
$(b_1,\dots, b_m)$  as in the diagram below:
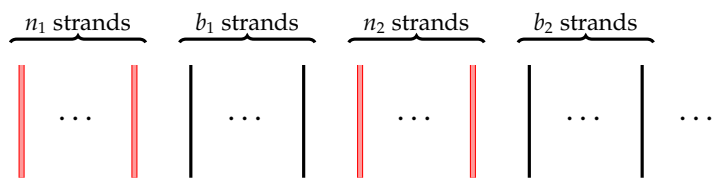
\begin{equation*}
  \tikz[very thick, scale=1.5]{
\draw[wei] (0,-.5) -- (0,.5);\draw[wei] (1,-.5) -- (1,.5);\node at (.5,0)
{$\cdots$}; \draw[decorate,decoration=brace,-] (-.1,.7) --
    node[above,midway,scale=.75]{$n_1$ strands} (1.1,.7);
\draw (1.5,-.5) -- (1.5,.5);\draw (2.5,-.5) -- (2.5,.5);\node at (2,0)
{$\cdots$}; \draw[decorate,decoration=brace,-] (1.4,.7) --
    node[above,midway,scale=.75]{$b_1$ strands} (2.6,.7);
\draw[wei] (3,-.5) -- (3,.5);\draw[wei] (4,-.5) -- (4,.5);\node at (3.5,0)
{$\cdots$}; \draw[decorate,decoration=brace,-] (2.9,.7) --
    node[above,midway,scale=.75]{$n_2$ strands} (4.1,.7);
\draw (4.5,-.5) -- (4.5,.5);\draw (5.5,-.5) -- (5.5,.5);\node at (5,0)
{$\cdots$}; \draw[decorate,decoration=brace,-] (4.4,.7) --
    node[above,midway,scale=.75]{$b_2$ strands} (5.6,.7);
\node at (6,0)
{$\cdots$};
}
\end{equation*}
The algebra $T^{\mathbf{n}}=e_{\mathbf{n}}T^\ell e_{\mathbf{n}}$ can
be represented using Stendhal diagrams as well, where we compress each
group of $n_i$ strands between which no blacks are allowed into a single
strand, labeled with $n_i$. This algebra naturally appears in the
construction of categorified colored Jones polynomials since its
Grothendieck group is a tensor product of simple $\mathfrak{sl}_2$
modules.  

 \begin{proposition}
   The horizontal composition $\mathsf{P}_{n_1}\otimes \cdots \otimes
   \mathsf{P}_{n_m}$ of 1-morphisms in ${\mathcal{BN}}$ is sent by $\gamma$ to the projection $T^\ell
   e_{\mathbf{n}} T^\ell\Lotimes -$.  
 \end{proposition}
 \begin{proof}
   Much like that of Theorem \ref{one-projector} above, the proof is by checking that both
   functors act by the identity on the subcategory generated by
   $T^\ell e_{\mathbf{n}} $ and trivially on its orthogonal.  

   The action on the subcategory generated by $T^\ell e_{\mathbf{n}} $
   can be understood by studying the actions on standardizations of
   projective-injectives of $T^{n_1}\otimes \cdots \otimes T^{n_m}$;
   this is the identity since \[\mathsf{P}_{n_1}\otimes \cdots \otimes
   \mathsf{P}_{n_m}\circ \mathbf{S}^{\mathbf{n}}\cong
   \mathbf{S}^{\mathbf{n}}\circ \mathsf{P}_{n_1}\otimes \cdots \otimes
   \mathsf{P}_{n_m}\] where $\mathbf{S}^{\mathbf{n}} $ is the
   standardization functor from \cite[\S 5]{Webmerged}. Since the projection on the right-hand side sends
   each projective-injective to itself, $\mathsf{P}_{n_1}\otimes \cdots \otimes
   \mathsf{P}_{n_m}$ must act by the identity on the category
   generated by these standardizations.

   On the other hand, the orthogonal to this category is generated by
   the images of cup diagrams with no cups that go between different
   groups of red strands.  These are killed by $\mathsf{P}_{n_1}\otimes \cdots \otimes
   \mathsf{P}_{n_m}$ by contractibility under turnbacks.
 \end{proof}
     The colored Jones homology theory in \cite[\S 8]{Webmerged} is defined
   using tensor product with certain bimodules corresponding to the
   cups, caps and crossings.  In fact their definition is essentially
   exactly like that of the functors $\mathbb{B}_i$, $\mathbb{K}_i$
   and $\mathbb E_{i}$ above.  

Let $\mathcal{T}$ be a tangle with components
   labeled by integers, and $\mathcal{T}'$ its cabling, with each strand
   replaced by as many strands as its label, as illustrated for a
   single crossing below.
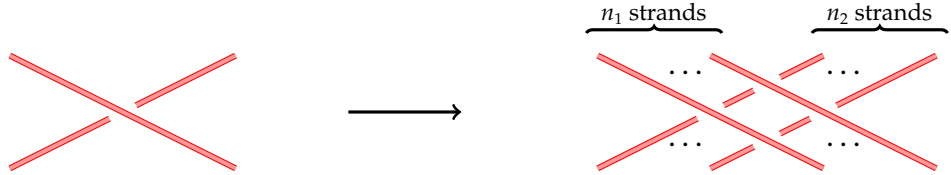
\begin{equation*}
  \tikz[very thick, scale=1.5]{
\draw[wei] (0,-.5) --  node [midway,circle, fill=white]{}  (2,.5);
\draw[wei] (2,-.5) --(0,.5); \draw[->] (3,0) -- (4,0); \node at (5,0) {\mbox{}};
}
\tikz[very thick, scale=1.5]{
\node at (.8,.32) {$\cdots$}; \node at (.8,-.32)
{$\cdots$}; \node at (2.2,.32) {$\cdots$}; \node at (2.2,-.32)
{$\cdots$};
\draw[decorate,decoration=brace,-] (-.1,.7) --
    node[above,midway,scale=.75]{$n_1$ strands} (1.1,.7);
\draw[wei] (0,-.5) --  node [pos=.75,circle, fill=white]{}  node [pos=.5,circle, fill=white]{}   (2,.5);\draw[wei] (1,-.5) -- node [pos=.25,circle, fill=white]{}  node [pos=.5,circle, fill=white]{}  (3,.5); \draw[decorate,decoration=brace,-] (1.9,.7) --
    node[above,midway,scale=.75]{$n_2$ strands} (3.1,.7);
\draw[wei] (2,-.5) -- (0,.5);\draw[wei] (3,-.5) -- (1,.5);
}
\end{equation*}
 Let $(n_1,\dots,n_m)$ and
   $(n_1',\dots, n'_{m'})$ be the sequence of labels at the top and bottom of the
   tangle.

   Then we have the functor attached to this tangle by the homology
   theory of \cite{Webmerged}, which we denote
   $\mathscr{K}^{\mathbf{n}}(\mathcal{T})$, and the functor $\mathscr{K}(\mathcal{T}')$ attached to the
   cabling by the theory we have discussed in Sections \ref{sec:braiding}--\ref{sec:comp-with-khov}.  Assume now that $\mathcal{T}$
   is a single crossing, cup or cap. 
   \begin{lemma}\label{bimodules}
     As bimodules over $T^{\mathbf{n}}$ and $T^{\mathbf{n}'}$, 
     $\mathscr{K}^{\mathbf{n}}(\mathcal{T})$ and $e_{\mathbf{n}'}\mathscr{K}(\mathcal{T}')e_{\mathbf{n}}$ are isomorphic.
   \end{lemma}
   \begin{proof}
     For the braiding map, this follows from the same argument as in
     \cite[4.19]{Webmerged}.  For the cup and cap, these are
     equivalent so we need only consider one.  The cabling of the cup
     is $n$ nested cups.  As usual, by considering the action on
     standardizations, we can reduce to the case where there are not
     any other red strands.  
    
     In this case, we need to show that these nested cups give us the
     unique invariant simple for the algebra $T^{(n,n)}\cong e_{(n,n)}T^{2n}e_{(n,n)}$ after being multiplied by
     $e_{(n,n)}$.  Multiplying by this idempotent is an exact functor,
     and it categorifies the projection $(\C^2)^{\otimes 2\ell}\to
     \operatorname{Sym}^\ell(\C^2)\otimes
     \operatorname{Sym}^\ell(\C^2)$.  In particular, it sends the
     image of nested cups to an invariant vector in $\operatorname{Sym}^\ell(\C^2)\otimes
     \operatorname{Sym}^\ell(\C^2)$ which is the class of invariant
     simple.  We can check this by looking at the coefficient of any
     monomial in the class, for example that of \[\left[\begin{smallmatrix}
  0\\1
\end{smallmatrix}\right]\otimes\cdots \otimes \left[\begin{smallmatrix}
  0\\1
\end{smallmatrix}\right]\otimes \left[\begin{smallmatrix}
  1\\0
\end{smallmatrix}\right]\otimes\cdots \otimes \left[\begin{smallmatrix}
  1\\0
\end{smallmatrix}\right]\] and checking that it is 1.  

Thus $e_{(n,n)}L_{n,n}$ is an honest module whose class in the
Grothendieck group coincides with the correct simple.  This is only
possible if it is the desired simple itself.
   \end{proof}
 \begin{corollary}
   The colored Jones homology theories defined in \cite{Webmerged} and
   \cite{CoKr} agree.
 \end{corollary}
 \begin{proof}
By its definition, the homology theory from \cite{CoKr} can be
obtained by taking a generic tangle projection, sliced into crossings,
cups, and caps; we'll use {\it cuts} to mean the horizontal lines
where we cut, and {\it slices} to mean the regions between two
successive ones. We let $\mathcal{T}_k$ be the tangle in the $k$th
slice from the bottom, and $\mathbf{n}_k$ the labeling of the strands
at the $k$th slice.

Now, we take this tangle's cabling, and insert a copy of
$\mathsf{P}_{n}$ at each point where a strand of label $n$ crosses one
of the horizontal cuts.  The image of this 1-morphism in Bar-Natan's category is
obtained by applying $\mathscr{K}(\mathcal{T}_k)$ for the different slices
$\mathcal{T}_k$ of the
cabled tangle with $T^\ell
   e_{\mathbf{n}_k} T^\ell\Lotimes -$ inserted at the $k$th cut.  We can do
   the 
   factorization $T^\ell
   e_{\mathbf{n_k}} T^\ell\Lotimes -\cong T^\ell
   e_{\mathbf{n_k}} \Lotimes_{T^{\mathbf{n}_k}} e_{\mathbf{n_k}}
   T^\ell\otimes -$ at each cut, and move the first factor into the
   slice above the cut, and the second factor into the slice below
   it.  Thus, for each slice $\mathcal{T}_{k}$, we obtain
   $e_{\mathbf{n}_{k-1}}\mathscr{K}(\mathcal{T}_k)e_{\mathbf{n}_k}\cong
   \mathscr{K}^{\mathbf{n}}(\mathcal{T}_k)$.  By definition, taking this
   successive derived tensor product gives the homology theory from \cite{Webmerged}.
 \end{proof}
Khovanov has also defined a categorification of the colored Jones
polynomial \cite{KhCJ}; this cannot agree with the theory defined
above, since it is finite dimensional.  The relation between these
theories seems to not be well-known at the moment.

 \bibliography{../gen}
\bibliographystyle{amsalpha}

\end{document}